\documentclass[reqno]{amsart}
\makeatletter
\def\oversortoftilde#1{\mathop{\vbox{\m@th\ialign{##\crcr\noalign{\kern3\p@}%
      \sortoftildefill\crcr\noalign{\kern3\p@\nointerlineskip}%
      $\hfil\displaystyle{#1}\hfil$\crcr}}}\limits}

\def\sortoftildefill{$\m@th \setbox\z@\hbox{$\braceld$}%
  \braceld\leaders\vrule \@height\ht\z@ \@depth\z@\hfill\braceru$}

\makeatother

\usepackage{amscd, amsfonts, amsmath,amssymb,amsthm, mathtools} 
\usepackage{tabularx,ltablex}
\usepackage[all,color]{xy}
\usepackage[hidelinks]{hyperref}
\usepackage{cleveref}
\usepackage{float}
\usepackage{changes}
\usepackage{mathrsfs} 
\usepackage{amsmath,latexsym, bbold}
\usepackage{endnotes}
\usepackage{ulem}
\usepackage{aligned-overset}
\usepackage[margin=1in]{geometry}

\allowdisplaybreaks

\newtheorem{thm}{Theorem}[section]

\newtheorem*{maintheorem*}{Main Theorem}
\newtheorem{defn}[thm]{Definition}
\newtheorem{Ex}[thm]{Example} 
\newtheorem{lemma}[thm]{Lemma} 
\newtheorem{proposition}[thm]{Proposition} 
\newtheorem{remark}[thm]{Remark} 
\newtheorem{Cor}[thm]{Corollary}
\newtheorem{ques}[thm]{Question}

\DeclareMathOperator{\id}{id}

\DeclareMathOperator{\Hom}{Hom}

\DeclareMathOperator{\comod}{comod}

\DeclareMathOperator{\Sh}{Sh}

\newcommand{\uaut}{\underline{\rm aut}}

\DeclareMathOperator{\grmod}{grmod}

\usepackage{tikz}
\usetikzlibrary{matrix}
\usepackage{graphicx,ctable,booktabs}
\usetikzlibrary{arrows.meta}
\usepackage{tikz-cd}

\newcommand{\kk}{\Bbbk}

\newcommand{\mc}{\mathcal}
\newcommand{\uend}{\underline{\rm end}}

\def\vps{\varepsilon}

\numberwithin{equation}{section}

\title{Quantum-symmetric equivalence is a graded Morita invariant}

\author[Huang]{Hongdi Huang}
\address{(Huang) Department of Mathematics, Shanghai University, Shanghai, 200444, China}
\email{hdhuang@shu.edu.cn}

\author[Nguyen]{Van C. Nguyen}
\address{(Nguyen) Department of Mathematics, United States Naval Academy, Annapolis, MD 21402, U.S.A.}
\email{vnguyen@usna.edu}

\author[Vashaw]{Kent B. Vashaw}
\address{(Vashaw) Department of Mathematics,
University of California Los Angeles,
Los Angeles, CA 90095, U.S.A.}
\email{kentvashaw@math.ucla.edu}

\author[Veerapen]{Padmini Veerapen}
\address{(Veerapen) Department of Mathematics, Tennessee Tech University, Cookeville, TN 38505, U.S.A.}
\email{pveerapen@tntech.edu}

\author[Wang]{Xingting Wang}
\address{(Wang) Department of Mathematics, Louisiana State University, Baton Rouge, Louisiana 70803, USA}
\email{xingtingwang@math.lsu.edu}

\date\today

\begin{document}

%%%%%%%%%%%%%%%%%%%%%%%%%%%%%%%%%%%
\begin{abstract}
    We show that if two $m$-homogeneous algebras have Morita equivalent graded module categories, then they are quantum-symmetrically equivalent, that is, there is a monoidal equivalence between the categories of comodules for their associated universal quantum groups (in the sense of Manin) which sends one algebra to the other. As a consequence, any Zhang twist of an $m$-homogeneous algebra is a 2-cocycle twist by some 2-cocycle from its Manin's universal quantum group. 
\end{abstract}
%%%%%%%%%%%%%%%%%%%%%%%%%%%%%%%%%%%

\maketitle

%\tableofcontents   

%%%%%%%%%%%%%%%%%%%%%%%%%%%%%%%%%%%
%%%%%%%%%%%%%%%%%%%%%%%%%%%%%%%%%%%
\section{Introduction}

Symmetry has been a central topic of study in mathematics for thousands of years. Symmetries of classical objects form a group; however, some quantum objects exhibit properties that cannot be captured by classical symmetries. This motivates the study of their quantum symmetries, which are better described by group-like objects known as quantum groups, whose representation categories provide examples of tensor categories (see e.g., \cite{EGNO}).

In his seminal work \cite{Manin2018}, Manin restored the ``broken symmetry" of a quantized algebra by imposing some non-trivial relations on the coordinate ring of the general linear group. This led to the introduction of the now-called ``Manin's universal quantum group".

\begin{defn}\cite[Lemma 6.6]{Manin2018}
\label{def:ManinU} 
Let $A$ be any $\mathbb Z$-graded locally finite $\kk$-algebra. The \emph{right universal bialgebra $\underline{\rm end}^r(A)$ associated to $A$} is the bialgebra that right coacts on $A$ preserving the grading of $A$ via $\rho: A\to A \otimes \underline{\rm end}^r(A)$ satisfying the following universal property: if $B$ is any bialgebra that right coacts on $A$ preserving the grading of $A$ via $\tau: A\to A \otimes B$, then there is a unique bialgebra map $f: \underline{\rm end}^r (A)\to B$ such that the diagram 
\begin{align}
\label{def:aut}
\xymatrix{
A\ar[r]^-{\rho}\ar[dr]_-{\tau} & A \otimes \underline{\rm end}^r(A) \ar[d]^-{\id \otimes f} \\
& A \otimes B
}
\end{align}
commutes. By replacing ``bialgebra" with ``Hopf algebra" in the above definition, we define the right universal quantum group $\uaut^r(A)$ to be the universal Hopf algebra right coacting on $A$.
\end{defn} 

\begin{remark}
    One can also define a left-coacting version of Manin's universal quantum groups. All results in this paper can be proven analogously in that context.
\end{remark}

There is a current surge of interest in the study of universal quantum symmetries, see e.g., \cite{Agore2021,AGV,Chirvasitu-Walton-Wang2019,HNUVVW21,HNUVVW2,HNUVVW3, HWWW,vdb2017,WaltonWang2016}. Notable results by Raedschelders and Van den Bergh in \cite{vdb2017} showed that Manin's universal quantum groups of Koszul Artin-Schelter (AS) regular algebras with the same global dimensions have monoidally equivalent comodule categories. In \cite{HNUVVW3}, the authors together with Ure introduced quantum-symmetric equivalence to systematically study such algebras. 

\begin{defn}
\label{DefnA}\cite[Definition A]{HNUVVW3}
Let $A$ and $B$ be two connected graded algebras finitely generated in degree one. We say $A$ and $B$ are \emph{quantum-symmetrically equivalent} if there is a monoidal equivalence between the comodule categories of their associated universal quantum groups 
\[
\comod(\uaut^r(A))~\overset{\otimes}{\cong}~\comod(\uaut^r(B))
\]
in the sense of Manin, where this equivalence sends $A$ to $B$ as comodule algebras. We denote the quantum-symmetric equivalence class of $A$ by $QS(A)$, which consists of all connected graded algebras that are quantum-symmetrically equivalent to $A$.
\end{defn}

For any connected graded algebra $A$ finitely generated in degree one, we aim to determine its $QS(A)$. The main findings in \cite{HNUVVW3} demonstrate that all graded algebras in $QS(A)$ have various homological properties in common with $A$, and that the family of Koszul AS-regular algebras of a fixed global dimension forms a single quantum-symmetric equivalence class.

The purpose of this paper is to explore additional properties of $A$ that may help to identify characteristics of $QS(A)$ beyond the numerical and homological invariants explored in \cite{HNUVVW3}. It is important to note that these numerical and homological invariants of $A$ are entirely determined by its graded module category $\grmod(A)$. Moreover, in \cite{Zhang1996}, Zhang fully characterized the graded Morita equivalence between two $\mathbb N$-graded algebras by Zhang twists given by some twisting systems (see \cite{Sierra2011} for a generalization to $\mathbb{Z}$-graded algebras and \cite{LopezWalton} for a generalization to algebras in monoidal categories). Therefore, we pose a natural question: Does $\grmod(A)$ uniquely determine $QS(A)$? Our main finding in this paper answers this question positively in the case of all $m$-homogeneous algebras.

\begin{thm}
    \label{atau-qsequiv}
    For any integer $m \geq 2$, let $A$ and $B$ be two $m$-homogeneous algebras. If $A$ and $B$ are graded Morita equivalent, then they are quantum-symmetrically equivalent. 
\end{thm}
In particular, we show that a Zhang twist of an $m$-homogeneous algebra by a twisting system can be realized as a 2-cocycle twist by using its universal quantum group $\uaut^r(A)$. A base case of this result, when the twisting system is formed by the compositions of a single algebra automorphism, was achieved in \cite[Theorem 2.3.3]{HNUVVW3}; the present generalization to arbitrary twisting systems involves significant technical complications and applies in much greater generality (see, e.g., \cite{TVan}).

%%%%%%%%%%%%%%%%%%%%%%%%%%%%
\subsection*{Conventions} 

Throughout, let $\kk$ be a base field with $\otimes$ taken over $\kk$ unless stated otherwise. A $\mathbb Z$-graded algebra $A=\bigoplus_{i\in \mathbb Z} A_i$ is called \emph{connected graded} if $A_i=0$ for $i<0$ and $A_0=\kk$. For any integer $m \geq 2$, an \emph{$m$-homogeneous algebra} is a connected graded algebra $A:=\kk\langle A_1 \rangle/(R)$ finitely generated in degree one, subject to $m$-homogeneous relations $R \subseteq A_1^{\otimes m}$. For any homogeneous element $a\in A$, we denote its degree by $|a|$. We use the Sweedler notation for the coproduct in a coalgebra $B$: for any $h \in B$, $\Delta(h) = \sum h_1 \otimes h_2 \in B \otimes B$. The category of right $B$-comodules is denoted by $\comod(B)$. 

\subsection*{Acknowledgements} 
The authors thank the referee for their careful reading and suggestions to improve the paper. Some of the results in this paper were formulated at a SQuaRE at the American Institute of Mathematics; the authors thank AIM for their hospitality and support. Nguyen was partially supported by NSF grant DMS-2201146. Vashaw was partially supported by NSF Postdoctoral Fellowship DMS-2103272. Veerapen was partially supported by an NSF--AWM Travel Grant.

%%%%%%%%%%%%%%%%%%%%%%%%%%%%%%%%%%%
%%%%%%%%%%%%%%%%%%%%%%%%%%%%%%%%%%%
\section{Lifting twisting systems to universal bialgebras}
\label{sect:genresults}

For any $\mathbb Z$-graded algebra $A$, recall that a \emph{twisting system} of $A$ consists of a collection $\tau:= \{\tau_i: i \in \mathbb Z\}$ of $\mathbb Z$-graded bijective linear maps $\tau_i: A \to A$, satisfying any one of the following equivalent conditions (see \cite[(2.1.1)-(2.1.4)]{Zhang1996}):
\begin{enumerate}
    \item $\tau_i(a \tau_j(b)) = \tau_i(a) \tau_{i+j}(b)$;
    \item $\tau_i(ab)=\tau_i(a) \tau_{i+j} \tau_j^{-1}(b)$;
    \item $\tau_i^{-1}( a \tau_{i+j}(b))= \tau_i^{-1} (a) \tau_j(b)$;
    \item $\tau_i^{-1}(ab) = \tau_i^{-1} (a) \tau_j \tau_{i+j}^{-1}(b)$,
\end{enumerate}
for homogeneous elements $a, b \in A$, where $a$ is of degree $j$ and $b$ is of any degree.
By \cite[Proposition 2.4]{Zhang1996}, we may always assume the following additional two conditions:
\begin{itemize}
\item[(5)] $\tau_i(1)=1$ for any $i\in \mathbb Z$;
\item[(6)] $\tau_0={\rm id}_A$.
\end{itemize}
% When $\tau$ satisfies (1) (or, equivalently (2)), we say that it satisfies the {\emph{twisting systems axioms}}. \sout{and when  $\tau^{-1}$ satisfies (3) or (4) we say it satisfies the {\emph{inverse twisting system axioms}}.}

For any twisting system $\tau$ of $A$, the \emph{right Zhang twist of $A$}, denoted by $A^\tau$, is the graded algebra such that $A^\tau=A$ as graded vector spaces with the twisted product $a \cdot_\tau b=a\tau_j(b)$, for homogeneous elements $a,b \in A$, where $a$ is of degree $j$ and $b$ is of any degree.

For an $m$-homogeneous algebra $A=\kk\langle A_1 \rangle/(R)$, we construct a twisting system of $A$ explicitly as follows. Let $\tau:=\{\tau_i: A_1\to A_1\}_{i \in \mathbb{Z}}$ be a collection of bijective linear maps on degree one (where $\tau_0=\id$) with $\kk$-linear inverses $\tau^{-1}:=\{\tau_i^{-1}: A_1\to A_1\}_{i\in \mathbb Z}$. We extend each $\tau_i$ and $\tau_i^{-1}$ (which we denote as $\tau_i$ and $\widetilde{\tau}_i$, respectively, by abuse of notation) to $\kk\langle A_1\rangle$ inductively on the total degree of the element $ab$ by the rules:  
\begin{equation}
\label{eq:inductive}
\tau_i(1)=\widetilde{\tau_i}(1)=1, \qquad  
\tau_i(ab):= \tau_i(a) \tau_{i+1} \widetilde \tau_1 (b),
\qquad \text{ and } \qquad
\widetilde \tau_i(ab):= \widetilde \tau_i(a)\tau_1 \widetilde\tau_{i+1}(b),  
\end{equation}
for any $a \in A_1$ and $b$ is of any positive degree. In the following result, we use the rules in \eqref{eq:inductive} to define a twisting system of $A$ by proving that $\tau_i$ and $\widetilde \tau_i$ indeed satisfy the twisting system axioms (with $\widetilde \tau_i$ being the inverse of $\tau_i$) if and only if they preserve the relation space $R$ of $A$.

\begin{proposition}
\label{twist-deg-1}
    Let $A=\kk\langle A_1\rangle/(R)$ be an $m$-homogeneous algebra and $\tau_i$ and $\widetilde \tau_i$ be defined as in \eqref{eq:inductive}. If $\tau_i(R) = R$ for all $i \in \mathbb{Z}$, then $\tau_i$ and $\widetilde \tau_i$ are well-defined graded linear maps $A \to A$ that are inverse to each other. Moreover, $\tau=\{\tau_i: i \in \mathbb Z\}$ is a twisting system of $A$.  
\end{proposition}

\begin{proof}
By assumption, it is clear that $\tau_i$ is well-defined and bijective on all degrees up to and including $m$, and that $\widetilde \tau_i$ is well-defined on all degrees less than $m$, and is inverse to $\tau_i$ on degree 1 by definition. Furthermore, again by definition, $\tau$ and $\widetilde \tau$ satisfy the twisting system axioms and inverse twisting system axioms, respectively, on degrees $\leq 2$. We now show inductively on arbitrary degree $n$ that $\tau$ and $\widetilde \tau$ are well-defined, bijective, inverse to each other, and satisfy the (inverse) twisting system axioms.

We first show that $\tau$ satisfies the twisting systems axioms on the free algebra $\kk\langle A_1\rangle$. Suppose that $a$ and $b$ are homogeneous monomial elements of degrees $j$ and $n-j$, respectively. Assume that $a=a_1 a_2$ for $a_1$ of degree 1 and $a_2$ of degree $j-1$; note that elements of this form span $A_j$, since we assume that $A$ is generated in degree 1. Then for all $i \in \mathbb Z$, we have
    \begin{align*}
        \tau_i(ab)&= \tau_i(a_1 a_2 b)\\
        &= \tau_i(a_1) \tau_{i+1}  \widetilde \tau_1 (a_2 b)\\
        &=\tau_i(a_1) \tau_{i+1}\big( \widetilde \tau_1(a_2) \tau_{j-1} \widetilde \tau_j(b)\big)\\
        &=\tau_i(a_1) \tau_{i+1} \widetilde \tau_1(a_2) \tau_{i+j} \widetilde \tau_{j-1}\tau_{j-1} \widetilde \tau_j(b)\\
        &=\tau_i(a_1) \tau_{i+1}\widetilde \tau_1(a_2) \tau_{i+j}\widetilde \tau_j(b)\\
        &=\tau_i(a_1 a_2) \tau_{i+j}\widetilde \tau_j(b)\\
        &=\tau_i(a) \tau_{i+j}\widetilde \tau_j(b).
    \end{align*}
The second equality is from the definition of $\tau_i$ in \eqref{eq:inductive}. The third, fourth, and sixth equalities follow from the inductive hypothesis as $\tau$ and $\widetilde \tau$ satisfy the (inverse) twisting axioms up to degrees $< n$.  Thus, $\tau$ satisfies the twisting system axioms. An analogous argument shows that $\widetilde \tau$ satisfies the inverse twisting system axioms. Moreover, we note that $\tau_i$ and $\widetilde \tau_i$ are inverse to one another on $\kk\langle A_1\rangle$ by induction since 
\[
    \widetilde \tau_i \tau_i (a b) = \widetilde \tau_i (\tau_i(a) \tau_{i+1} \widetilde \tau_1 (b)) = \widetilde \tau_i\tau_i (a) \tau_1\widetilde \tau_{i+1} \tau_{i+1} \widetilde \tau_1 (b) = ab,
\]
for any $a \in A$ of degree 1 and $b \in A$ of degree $n-1$.
    
It remains to show that for any $i \in \mathbb Z$, $\tau_i$ preserves the homogeneous relation ideal $(R)$ of $A$ in $\kk\langle A_1\rangle$. It is trivial for relations of degree $n\leq m$. An arbitrary relation of degree $n>m$ is a linear combination of terms of the form $r a$ and $a r$, where $a$ is an element of degree $1$ in $ A$ and $r$  is a relation of degree $n-1$. But note that $\tau_i(a r)$ is indeed a relation of $A$, since $\tau_i(a r) = \tau_i(a) \tau_{i+1} \widetilde \tau_1(r)$ by the twisting system axioms, and $\tau_{i+1} \widetilde \tau_1(r)$ is a relation of $A$ by the inductive hypothesis. Similarly, $\tau_i$ sends $ra$ to a relation of $A$, so $\tau_i$ preserves all homogeneous relations of degree $n$. This completes the proof.
\end{proof}

%%%%%%%%%%%%%%%%%%%%%%%%%%%%%%%%%

Recall that the \emph{Koszul dual} of an $m$-homogeneous algebra $A=\kk \langle A_1\rangle/(R)$ is the $m$-homogeneous algebra
\[
A^! := \kk \langle A_1^*\rangle / (R^\perp),
\]
where $A_1^*$ is the vector space dual of $A_1$ and $R^\perp \subseteq (A_1^*)^{\otimes m}$ is the subspace orthogonal to $R$ with respect to the natural evaluation $\langle-,-\rangle: A_1^*\times A_1\to \kk$. 

Let $\tau=\{\tau_i: i \in \mathbb{Z}\}$ be a twisting system of $A$ with inverse twisting system $\{\tau_i^{-1}: i\in \mathbb Z\}$. We define the dual twisting system $\tau^!=\{\tau_i^!: i\in \mathbb Z\}$ together with the inverse dual twisting system $(\tau^!)^{-1}=\{(\tau_i^!)^{-1}: i\in \mathbb Z\}$ on the Koszul dual $A^!$ such that 
\[\tau_i^!|_{A^!_1} :=(\tau_i^{-1})^* = (\tau_i^*)^{-1} \qquad \text{ and } \qquad 
(\tau_i^!)^{-1}|_{A_1^!} := \tau_i^*,\] 
as linear maps $A_1^* \to A_1^*$. 
For $a \in A^!_1$ and $b \in A^!$ is of any positive degree, we define each $\tau_i^!$ and $(\tau_i^!)^{-1}$ inductively on the total degree of the element $ab$ as follows: 
\begin{equation}
\label{dual twisting}
\tau_i^! (ab)=\tau_i^!(a) \tau_{i+1}^! (\tau^{-1}_1)^!(b),
\qquad \text{and} \qquad 
(\tau_i^!)^{-1}(ab)=(\tau_i^!)^{-1} (a) \tau_1^! (\tau_{i+1}^!)^{-1}(b).
\end{equation}
Using \Cref{twist-deg-1}, in the following we show that these maps give well-defined twisting systems of $A^!$.

\begin{proposition}
\label{prop:twistingtau!}
Let $A, A^!$ and $\tau^{\pm 1},$ $(\tau^!)^{\pm 1}$ be defined as above. The collection of linear maps $\tau^!$, defined in \eqref{dual twisting}, forms a twisting system of $A^!$ with inverse  $(\tau^!)^{-1}$. 
\end{proposition}

\begin{proof}
By \Cref{twist-deg-1}, it is enough to show that $\tau_i^!(R^\perp)= R^\perp$. We first inductively show that 
\begin{equation}
    \label{eq:tauwelldefined}
\langle \tau_i^!(f), a\rangle =\langle f, \tau_i^{-1}(a)\rangle 
\qquad \textnormal{ and } \qquad \langle (\tau^{!}_i)^{-1}(f), a\rangle =\langle f, \tau_i(a)\rangle 
\end{equation}
for $f \in (A_1^{\otimes n})^*$ and $a \in  A_1^{\otimes n}$ for $n\ge 1$. The case $n=1$ follows from the definition. Assume the inductive hypothesis, we now show \eqref{eq:tauwelldefined} holds for $n+1$. Without loss of generality, let $f=yg$ and $a=xh$ for any $y\in A_1^*,g\in (A_1^*)^{\otimes n}$ and $ x\in A_1, h\in (A_1)^{\otimes n}$. Then we have 
    \begin{align*}
        \langle \tau_i^!(f), a\rangle  &= \langle \tau_i^!(y)\tau_{i+1}^! \tau_{i+1}^!(\tau^!_1)^{-1}(g), xh\rangle \\
        &=\langle \tau_i^! (y), x\rangle \langle \tau_{i+1}^!(\tau^!_1)^{-1}(g), h\rangle \\
        &= \langle y, \tau_i^{-1} (x)\rangle \langle g, \tau_1 \tau_{i+1}^{-1} (h)\rangle \\
        &= \langle yg, \tau_i^{-1} (x)\tau_1 \tau^{-1}_{i+1} (h)\rangle \\
        &=\langle f, \tau_i^{-1}(a)\rangle,
    \end{align*}
where the last equality follows from the fact that $\tau$ is a twisting system. By a straightforward induction, it similarly follows that $\langle (\tau_i^!)^{-1}(f), a\rangle =\langle f, \tau_i(a)\rangle$. So we have
$\tau^!_i(R^\perp)=R^\perp\Leftrightarrow \langle \tau_i^!(R^\perp), R\rangle =0 \Leftrightarrow \langle R^\perp, (\tau_i^{-1})(R)\rangle =0\Leftrightarrow\tau^{-1}_i(R)=R\Leftrightarrow \tau_i(R)=R$, which holds by assumption. It follows that $\tau^!$ is a twisting system of $A^!$.
\end{proof}

\begin{proposition}
\label{prop:tau!}
Let $A$ be an $m$-homogeneous algebra with a twisting system $\tau=\{\tau_i: i\in \mathbb Z\}$. Then $(A^!)^{\tau^!}=(A^{\tau})^!$.    
\end{proposition}

\begin{proof}
Write $A=\kk \langle A_1\rangle/(R)$ with $m$-homogeneous relations $R\subseteq A_1^{\otimes m}$. By \Cref{prop:twistingtau!}, $\tau^!=\{\tau_i^!: i\in \mathbb Z\}$ is a twisting system of $A^!$. Similar to \cite[Lemma 5.1.1]{Mori-Smith2016}, one can check that $A^{\tau}=\kk \langle A_1\rangle/(R^{\tau})$, where $R^{\tau}=(\id \otimes \tau_1^{-1}\otimes \tau_2^{-1}\otimes \cdots \otimes \tau_{m-1}^{-1})(R)$. Notice that 
\[0=\langle (R^{\tau})^{\perp}, R^{\tau}\rangle=\langle (R^\tau)^\perp, (\id \otimes \tau_1^{-1}\otimes \tau_2^{-1}\otimes \cdots \otimes \tau_{m-1}^{-1})(R)\rangle =\langle (\id \otimes \tau_1^!\otimes \cdots\otimes \tau_{m-1}^!)(R^\tau)^\perp, R\rangle. \] Hence, $R^\perp=(\id \otimes \tau_1^!\otimes \cdots\otimes \tau_{m-1}^!)(R^\tau)^\perp$ and so \[(R^\perp)^{\tau^!}=(\id \otimes (\tau^{-1}_1)^!\otimes \cdots\otimes (\tau_{m-1}^{-1})^!)(\id\otimes \tau_1^!\otimes \cdots \otimes \tau_{m-1}^!)(R^\tau)^\perp=(R^{\tau})^\perp.\] 
As a result, we have 
$(A^{\tau})^!=(\kk \langle A_1\rangle/R^{\tau})^!=\kk \langle A_1^* \rangle/((R^{\tau})^{\perp})=\kk \langle A_1^* \rangle/((R^\perp)^{\tau^!})=(A^!)^{\tau^!}.$
\end{proof}
Let $V, W$ be any two finite-dimensional vector spaces. For any integer $m \geq 1$, we denote the {\emph shuffle map}
\[
\Sh_{V,W,m}: V^{\otimes m} \otimes W^{\otimes m} \xrightarrow{\cong} (V\otimes W)^{m} 
\]
to be the map sending
\[
v_1 \otimes v_2 \otimes ... \otimes v_m \otimes w_1 \otimes... \otimes w_m \mapsto v_1 \otimes w_1 \otimes v_2 \otimes w_2\otimes...\otimes v_m \otimes w_m, 
\]
for any $v_i \in V$ and $w_j \in W$. When $V$, $W$, and $m$ are clear from context, we omit the subscripts and denote this map by $\Sh$.

For two connected graded algebras $A=\kk\langle A_1\rangle/(R(A))$ and $B=\kk\langle B_1\rangle/(R(B))$ with $m$-homogeneous relations $R(A)\subseteq (A_1)^{\otimes m}$ and $R(B)\subseteq (B_1)^{\otimes m}$ respectively, we extend Manin's bullet product \cite[\S 4.2]{Manin2018} to $A$ and $B$ such that
\[
A\bullet B~:=~\frac{\kk \langle A_1\otimes B_1\rangle}{\left(\Sh(R(A)\otimes R(B))\right)},
\]
where $\Sh: (A_1)^{\otimes m} \otimes (B_1)^{\otimes m}\to (A_1\otimes B_1)^{\otimes m}$ is the shuffle map.  
When $B=A^!=\kk\langle A_1^*\rangle/(R(A)^\perp)$ is the $m$-Koszul dual algebra of $A$, by the definition of the bullet product we see that $A\bullet A^!$ is a connected graded bialgebra with matrix comultiplication defined on the generators of $A_1\otimes A_1^*$. In particular, choose a basis $\{x_1,\ldots,x_n\}$ for $A_1$ and let $\{x^1,\ldots,x^n\}$ be the dual basis for $(A^!)_1= A_1^*$. Write $z_j^k=x_j\otimes x^k \in A_1\otimes A_1^*$ as the generators for $A\bullet A^!$. Then the coalgebra structure of $A\bullet A^!$ is given by
\[
\Delta(z_j^k)=\sum_{1\leq i\leq n} z_i^k\otimes z_j^i, \qquad \text{ and } \qquad \varepsilon(z_j^k)=\delta_{j,k},\quad \text{for any }  1 \leq j,k \leq n.
\]

The following result is a straightforward generalization of the quadratic case in \cite{Manin2018}, which describes Manin's universal bialgebra $\underline{\rm end}^r(A)$ and Manin's universal quantum group $\underline{\rm aut}^r(A)$ in terms of the bullet product of $A$ and its Koszul dual $A^!$. 

\begin{lemma}\cite[Lemma 2.1.5]{HNUVVW3}
\label{lem:ManinM}
Let $A$ be an $m$-homogeneous algebra and $A^!$ be its Koszul dual. We have:
\begin{itemize}
    \item[(1)] $\underline{\rm end}^r(A)\cong A\bullet A^!$; 
    \item[(2)] $\underline{\rm aut}^r(A)$ is the Hopf envelope of $\underline{\rm end}^r(A)$.
\end{itemize}
\end{lemma}

We now show that the bullet product of two twisting systems of $A$ and of $B$ is indeed a twisting system of $A \bullet B$. As a consequence, we can extend any twisting system of $A$ to a twisting system of its universal bialgebra $\underline{\rm end}^r(A)$.
\begin{proposition}
\label{bullet-twist}
    Let $A$ and $B$ be two $m$-homogeneous algebras. If $\tau=\{\tau_i : i \in \mathbb{Z}\}$ is a twisting system of $A$, and $\omega = \{\omega_i : i \in \mathbb{Z}\}$ is a twisting system of $B$, then there exists a twisting system $\tau \bullet \omega$ of the algebra $A \bullet B$, where $(\tau \bullet \omega)_i$ on the degree one space $(A \bullet B)_1 \cong A_1 \otimes B_1$ corresponds to the map $\tau_i \otimes \omega_i$. Furthermore, $(A \bullet B)^{ \tau \bullet \omega} \cong A^\tau \bullet B^\omega$ as $m$-homogeneous algebras.
\end{proposition}

\begin{proof}
We construct $\tau \bullet \omega$ by extending $\tau\bullet \omega$ to the free algebra $\kk\langle A_1\otimes B_1\rangle$ as in \eqref{eq:inductive}. We claim that
\[
(\tau_i \bullet \omega_i) (\Sh(a \otimes b)) = \Sh(\tau_i(a) \otimes \omega_i(b)) 
\]
for all $i \in \mathbb Z$, and $a \in A, b \in B$ are of the same degree $n$. It is trivial for $n=0,1$. By induction on $n$, suppose it holds for $n\ge 1$. We now show it holds for $n+1$. Without loss of generality, we take $a=xa'$ and $b=yb'$ with $x\in A_1$, $a' \in (A_1)^{\otimes n}$ and $y\in B_1$, $b' \in (B_1)^{\otimes n}$.  Then we have
\begin{align*}
(\tau_i\bullet \omega_i)(\Sh(a\otimes b))&=(\tau_i\otimes \omega_i)(x\otimes y)\Sh(a'\otimes b')\\
&=(\tau_i\otimes \omega_i)(x\otimes y)(\tau_{i+1}\tau_1^{-1}\otimes \omega_{i+1}\omega_1^{-1})(\Sh(a'\otimes b'))\\
&=(\tau_i\otimes \omega_i)(x\otimes y)\Sh(\tau_{i+1}\tau_1^{-1}(a')\otimes \omega_{i+1}\omega_1^{-1}(b'))\\
&=\Sh(\tau_i(x)\tau_{i+1}\tau_1^{-1}(a')\otimes \omega_i(y)\omega_{i+1}\omega_1^{-1}(b'))\\
&=\Sh(\tau_i(a)\otimes \omega_i(b)).
\end{align*}
This proves our claim. Denote the degree $m$ relations of $A$ by $R$ and the degree $m$ relations of $B$ by $S$. In particular, we have
\[
(\tau_i\bullet \omega_i)(\Sh(R\otimes S))=\Sh((\tau_i(R)\otimes \omega_i(S))=\Sh(R\otimes S).
\]
According to \Cref{twist-deg-1}, we know $\tau\bullet \omega$ is a well-defined twisting system of $A\bullet B$.

We now check the final claim (compare with \cite[Lemma 3.1.1]{HNUVVW21}). Denote by $R^\tau$ and $S^\omega$ the relation spaces of $A^\tau$ and $B^\omega$, respectively. Recall that we have
$R^\tau = (\id \otimes \tau_1^{-1} \otimes \tau_2^{-1} \otimes\cdots \otimes\tau_{m-1}^{-1}) (R),$
and $S^\omega$ can be presented likewise. Then the relations of $A^\tau \bullet B^\omega$ are precisely
\begin{align*}
\Sh(R^\tau \otimes S^\omega) &= \Sh((\id \otimes \tau_1^{-1} \otimes \tau_2^{-1} \otimes\cdots \otimes \tau_{m-1}^{-1}) (R) \otimes (\id \otimes \omega_1^{-1} \otimes \omega_2^{-1} \otimes\cdots \otimes \omega_{m-1}^{-1}) (S))\\
&=(\id \otimes \id \otimes \tau_{1}^{-1} \otimes \omega_{1}^{-1} \otimes \tau_{2}^{-1} \otimes \omega_{2}^{-1} \otimes\cdots \otimes \tau_{m-1}^{-1} \otimes \omega_{m-1}^{-1} )(\Sh(R \otimes S)). 
\end{align*}
The last equality gives the relations of $(A \bullet B)^{\tau \bullet \omega}$. Thus, $(A \bullet B)^{ \tau \bullet \omega} \cong A^\tau \bullet B^\omega$ as $m$-homogeneous algebras.
\end{proof}

\begin{Cor}
Let $A$ be an $m$-homogeneous algebra with a twisting system $\tau=\{\tau_i: i\in \mathbb Z\}$. Then $\tau \bullet \tau^!$ is a twisting system of $\underline{\rm end}^r(A)$, and $\underline{\rm end}^r(A)^{\tau \bullet \tau^!} \cong A^\tau \bullet (A^!)^{\tau^!} \cong A^\tau \bullet (A^\tau)^! \cong \underline{\rm end}^r(A^\tau)$ as graded algebras.  
\end{Cor}

\begin{proof}
This is a direct consequence of \Cref{bullet-twist} by letting $B=A^!$ and applying \Cref{lem:ManinM}(1) and \Cref{prop:tau!}. 
\end{proof}

%%%%%%%%%%%%%%%%%%%%%%%%%%%%%%%%%%%
%%%%%%%%%%%%%%%%%%%%%%%%%%%%%%%%%%%
\section{Systems of twisting functionals}
\label{sect:lintwist}
Throughout this section, let $B$ be a bialgebra satisfying the twisting conditions below.
\begin{defn}\cite[Definition B]{HNUVVW3}
\label{defn:twisting conditions}
A bialgebra $(B,M,u,\Delta,\vps)$ satisfies the \emph{twisting conditions} if 
\begin{enumerate}
\item[(\textbf{T1})] as an algebra $B=\bigoplus_{n\in \mathbb{Z}} B_n$ is $\mathbb{Z}$-graded, and
\item[(\textbf{T2})] the comultiplication satisfies $\Delta(B_n)\subseteq B_n\otimes B_n$ for all $n\in \mathbb Z$. 
\end{enumerate}
\end{defn}

Recall that the space of linear functionals $\Hom_\kk(B,\kk)$ on $B$ has an algebra structure under the \emph{convolution product} $*$ such that $f*g=(f\otimes g) \circ \Delta$ with unit $u\circ \vps$.

\begin{lemma}\label{Esigma}
Let $\alpha=\{\alpha_i: B \to \kk: i\in \mathbb Z\}$ be a collection of linear functionals on a bialgebra $B$ such that each $\alpha_i$ is convolution invertible with inverse denoted by $\alpha_i^{-1}$. Then the following conditions are equivalent for any homogeneous elements $a,b \in B$, where $a$ is of degree $j$ and $b$ is of any degree: 
\begin{itemize}
\item[(1)] $\sum \alpha_i(ab_1)\alpha_j(b_2)=\alpha_i(a)\alpha_{i+j}(b)$;
\item[(2)] $\alpha_i(ab)=\alpha_i(a)(\alpha_{i+j}*\alpha^{-1}_j)(b)$;
\item[(3)] $\alpha^{-1}_i(ab)=\alpha^{-1}_i(a)(\alpha_j*\alpha^{-1}_{i+j})(b)$;
\item[(4)] $\sum \alpha^{-1}_i(ab_1)\alpha_{i+j}(b_2)=\alpha^{-1}_i(a)\alpha_j(b)$.
\end{itemize}
\end{lemma}

\begin{proof} We use the properties of the counit $\vps$ to show the equivalence below.\\
(1)$\Rightarrow$(2): 
\begin{align*}
\alpha_i(ab)&=\sum\alpha_i(ab_1)\vps(b_2)=\sum \alpha_i(ab_1)\alpha_j(b_2)\alpha^{-1}_j(b_3) \\
&=\sum \alpha_i(a)\alpha_{i+j}(b_1)\alpha^{-1}_j(b_2) =\alpha_i(a)(\alpha_{i+j}*\alpha^{-1}_j)(b).
\end{align*}
(2)$\Rightarrow$(1): 
\begin{align*}
\sum\alpha_i(ab_1)\alpha_j(b_2)&=\sum \alpha_i(a)(\alpha_{i+j}*\alpha^{-1}_j)(b_1)\alpha_j(b_2) \\ 
&=\sum \alpha_i(a)\alpha_{i+j}(b_1)\alpha^{-1}_j(b_2)\alpha_j(b_3) =\alpha_i(a)\alpha_{i+j}(b).
\end{align*}
We can show that (3)$\Leftrightarrow$(4) similarly. \\
(2)$\Rightarrow$(3): 
\begin{align*}
\alpha^{-1}_i(ab)&=\sum\alpha^{-1}_i(a_3)\alpha_i(a_2)\alpha^{-1}_i(a_1b_1)(\alpha_{i+j}*\alpha^{-1}_j)(b_2)(\alpha_j*\alpha^{-1}_{i+j})(b_3)\\
&=\sum \alpha^{-1}_i(a_3)\alpha_i(a_2)(\alpha_{i+j}*\alpha^{-1}_j)(b_2)\alpha^{-1}_i(a_1b_1)(\alpha_j*\alpha^{-1}_{i+j})(b_3)\\
&=\sum \alpha^{-1}_i(a_3)\alpha_i(a_2b_2)\alpha^{-1}_i(a_1b_1)(\alpha_j*\alpha^{-1}_{i+j})(b_3)\\
&=\alpha^{-1}_i(a)(\alpha_j*\alpha^{-1}_{i+j})(b).
\end{align*}
(3)$\Rightarrow$(2): 
\begin{align*}
\alpha_i(ab)&=\sum\alpha_i(a_3)\alpha^{-1}_i(a_2)\alpha_i(a_1b_1)(\alpha_{j}*\alpha^{-1}_{i+j})(b_2)(\alpha_{i+j}*\alpha^{-1}_{j})(b_3)\\
&=\sum \alpha_i(a_3)\alpha^{-1}_i(a_2)(\alpha_{j}*\alpha^{-1}_{i+j})(b_2)\alpha_i(a_1b_1)(\alpha_{i+j}*\alpha^{-1}_{j})(b_3)\\
&=\sum \alpha_i(a_3)\alpha^{-1}_i(a_2b_2)\alpha_i(a_1b_1)(\alpha_{i+j}*\alpha^{-1}_{j})(b_3)\\
&=\alpha_i(a)(\alpha_{i+j}*\alpha^{-1}_{j})(b).
\end{align*}
\end{proof}

\begin{defn}
A collection of linear functionals $\alpha=\{\alpha_i: B \to \kk\}_{i\in \mathbb Z}$ on a bialgebra $B$ is called a \emph{system of twisting functionals on $B$} if each $\alpha_i$ satisfies the following:
\begin{itemize}
\item[(1)] $\alpha_i$ is convolution invertible with inverse $\alpha^{-1}_i$;
\item[(2)] $\sum \alpha_i(ab_1)\alpha_j(b_2)=\alpha_i(a)\alpha_{i+j}(b)$, for $a \in B$ is of degree $j$ and $b \in B$ is homogeneous of any degree; 
\item[(3)] $\alpha_i(1)=1$; and
\item[(4)] $\alpha_0=\vps$, the counit of $B$. 
\end{itemize}
\end{defn}

Before we provide an example of a system of twisting functionals, we need the following notions. For any linear map $\pi: B \to \kk$, we define a linear map $\Xi^l[\pi]: B \to B$ via 
\[\Xi^l[\pi] = M \circ (\pi \otimes \id) \circ \Delta, \quad \text{that is,} \quad \Xi^l[\pi](b)=\sum \pi(b_1)b_2, \quad \text{for } b \in B,\] 
where $M$ denotes the multiplication map. We call $\Xi^l[\pi]$ a \emph{left linear winding map}, since it extends the notion of \emph{left winding endomorphism} in \cite[\S 2]{BZ08}. 
Similarly, the \emph{right linear winding map} $\Xi^r[\pi]$ is defined by
\[\Xi^r[\pi] = M \circ (\id \otimes \pi) \circ \Delta, \quad \text{that is,} \quad \Xi^r[\pi](b)=\sum b_1\pi(b_2), \quad \text{for } b \in B.\]

If in addition $\pi: B\to \kk$ is convolution invertible with inverse $\pi^{-1}: B\to \kk$, one can check that the linear inverse of $\Xi^{l}[\pi]$ is $(\Xi^{l}[\pi])^{-1}=\Xi^{l}[ \pi^{-1}]$, making $\Xi^{l}[\pi]$ a bijective linear winding map. Analogously, $\Xi^r[\pi]$ is also a bijective linear winding map with linear inverse $\Xi^r[\pi^{-1}]$. 

\begin{Ex}
Let $B$ be a Hopf algebra with antipode $S$. Let $\phi: B\to \kk$ be any algebra map. The convolution inverse of $\phi$ is $\phi^{-1}=\phi\circ S$. Consider $\alpha=\{\alpha_i\}_{i\in \mathbb Z}$ where $\alpha_i=\phi *\cdots * \phi$ is the $i$th product of $\phi$ with itself with respect to the convolution product in $\Hom_\kk(B,\kk)$. It is straightforward to check that $\alpha$ is a system of twisting functionals on $B$. Moreover, the associated twisting system $\tau=\{\tau_i\}_{i\in \mathbb Z}$ where $\tau_i=\Xi^r(\alpha_i)=(\Xi^r(\phi))^i$ is the twisting system given by the right bijective linear winding map associated with $\phi$.  
\end{Ex}

Recall that the \emph{Hopf envelope} of a bialgebra $B$ is the unique Hopf algebra $\mathcal H(B)$ together with a bialgebra map $\iota_B: B\to \mathcal H(B)$ satisfying the following universal property: for any bialgebra map $f: B\to K$ where $K$ is another Hopf algebra,  there is a unique Hopf algebra map $g: \mathcal H(B)\to K$ such that  $f=g\circ \iota_B$. It is proved in \cite[Lemma 2.1.10]{HNUVVW21} that if $B$ satisfies the twisting conditions in \Cref{defn:twisting conditions} then so does $\mathcal H(B)$, and additionally $S(\mathcal H(B)_n)\subseteq \mathcal H(B)_{-n}$, for any $n \in \mathbb Z$. 
 
Now, we construct explicitly the Hopf envelope $\mathcal H(B)$ as in \cite[Theorem 2.6.3]{Pareigis} and \cite{Porst2012}, which grew out of Takeuchi's construction for coalgebras \cite{Takeuchi1971}. Consider a presentation $B\cong\kk\langle V\rangle/(R)$ as graded algebras, where $V$ is a subcoalgebra of $B$. We can extend the comultiplication $\Delta$ and counit $\vps$ to the free algebra $\kk\langle V\rangle $ as algebra maps, where $(R)$ is a homogeneous bi-ideal of $\kk\langle V\rangle$. In this case, $B$ satisfies the twisting conditions. Denote infinitely many copies of the generating space $V$ as $\{V^{(k)}=V\}_{k\ge 0}$ and consider  
\begin{equation}\label{eq:freeT}
T:=\kk\langle\, \oplus_{k \ge 0} V^{(k)}\rangle.
\end{equation}
Let $S$ be the anti-algebra map on $T$ with $S(V^{(k)})=V^{(k+1)}$ for any $k \ge 0$. Both algebra maps $\Delta: \kk\langle V\rangle \to \kk\langle V\rangle \otimes \kk\langle V\rangle$ and $\vps: \kk\langle V\rangle\to \kk$ extend uniquely to $T$ as algebra maps via identities $(S\otimes S)\circ \Delta=\Delta \circ S$ and $\vps\circ S=\vps$, which we still denote by $\Delta: T\to T\otimes T$ and $\vps: T\to \kk$. The Hopf envelope of $B$ has a presentation
\[
\mathcal H(B)=T/W,   
\]
where the ideal $W$ is generated by
\begin{equation}\label{eq:RHB}
 S^k(R), \quad (M \circ (\id \otimes S) \circ \Delta-u \circ \vps)(V^{(k)}), \quad \text{and} \quad (M \circ (S\otimes \id) \circ \Delta-u \circ \vps)(V^{(k)}),\quad \text{for all } k \ge 0.   
\end{equation}
One can check that $W$ is a Hopf ideal of $T$, and so the Hopf algebra structure maps $\Delta$, $\vps$, and $S$ of $T$ give a Hopf algebra structure on $\mathcal H(B)=T/W$. Finally, the natural bialgebra map $\iota_B: B\to \mathcal H(B)$ is given by the natural embedding $\kk\langle V\rangle \hookrightarrow T$ by identifying $V=V^{(0)}$.

Suppose $B=\kk\langle V\rangle/(R)$ and $\alpha:=\{\alpha_i: V\to \kk\}_{i \in \mathbb{Z}}$ is a collection of linear functionals (with $\alpha_0=\varepsilon$) on the subcoalgebra $V$ with convolution inverses $\alpha^{-1}:= \{\alpha_i^{-1}: V\to \kk\}_{i\in \mathbb Z}$. We extend each $\alpha_i$ and $\alpha_i^{-1}$ (which we denote again as $\alpha_i$ and $\alpha^{-1}_i$, by abuse of notation) to $\kk\langle V\rangle$ inductively by the rules 
\begin{equation}\label{eq:inductivelinear}
\alpha_i(1)=\alpha_i^{-1}(1)=1, \qquad \alpha_i(ab):= \alpha_i(a) (\alpha_{i+1}* \alpha_1^{-1})(b),
\qquad \text{and} \qquad
\alpha_i^{-1}(ab):=\alpha_i^{-1}(a)(\alpha_1*\alpha_{i+1}^{-1})(b),  
\end{equation}
for any $a\in V$ and $b\in V^{\otimes n}$ for $n\ge 1$. We leave the proof of the following result to the reader as it is similar to the proof of \Cref{twist-deg-1}.

\begin{proposition}
    \label{twistL-deg-1}
    Retain the above notation. If $\alpha_i(R) =0$ for all $i \in \mathbb{Z}$, then $\alpha_i$ and $\alpha_i^{-1}$, defined in \eqref{eq:inductivelinear}, are well-defined linear functionals on $B$ that are convolution inverse to each other. Moreover, $\alpha=\{\alpha_i: i \in \mathbb Z\}$ is a system of twisting functionals on $B$. 
\end{proposition}

When $B$ is a Hopf algebra, our next result shows how twisting functionals are valued at the antipodes.
 
\begin{lemma}
\label{lem:inverseSLinear}
Let $H$ be a Hopf algebra satisfying the twisting conditions. Let $\alpha=\{\alpha_i: i \in \mathbb Z\}$ be a system of twisting functionals on $H$ with convolution inverse $ \alpha^{-1}=\{\alpha_i^{-1}: i\in \mathbb Z\}$. For any $i \in \mathbb Z$, any homogeneous element $a\in H$ of degree $j$, and any $k\ge 0$, we have:
\begin{align}\label{eq:sigmaS}
  \alpha_i(S^k(a))=
  \begin{cases}
  \alpha_i(a), & \text{$k$ is even}\\
  (\alpha_{-j}*\alpha^{-1}_{i-j})(a), & \text{$k$ is odd}
  \end{cases}
  \quad \text{and}\quad 
    \alpha^{-1}_i(S^k(a))=
  \begin{cases}
  \alpha^{-1}_i(a), & \text{$k$ is even}\\
  (\alpha_{i-j}*\alpha^{-1}_{-j})(a), & \text{$k$ is odd}.
  \end{cases}
  \end{align} 
\end{lemma}
\begin{proof}
For any $i \in \mathbb Z$, we proceed by induction on $k$. If $k=0$, the statement is trivial. When $k=1$, we have
  \begin{align*}
  \alpha_i(S(a))&=\sum \alpha_i(S(a_1))(\alpha_{i-j}*\alpha^{-1}_{-j})(a_2)(\alpha_{-j}*\alpha^{-1}_{i-j})(a_3)\\
  &=\alpha_i\left(\sum S(a_1)a_2\right)(\alpha_{-j}*\alpha^{-1}_{i-j})(a_3)\\
  &=(\alpha_{-j}*\alpha^{-1}_{i-j})(a)
  \end{align*}
  and 
    \begin{align*}
  \alpha^{-1}_i(S(a))&=\sum \alpha^{-1}_i(S(a_1))(\alpha_{-j}*\alpha^{-1}_{i-j})(a_2)(\alpha_{i-j}*\alpha^{-1}_{-j})(a_3)\\
  &=\alpha^{-1}_i\left(\sum S(a_1)a_2\right)(\alpha_{i-j}*\alpha^{-1}_{-j})(a_3)\\
  &=(\alpha_{i-j}*\alpha^{-1}_{-j})(a).
  \end{align*}
  Inductively for $\alpha_i(S^{k+1}(a))$, we have
  \[
  \alpha_i(S^{k+1}(a))=\alpha_iS(S^k(a))=(\alpha_{-j}*\alpha^{-1}_{i-j})(S^k(a))=\sum \alpha_{-j}(S^k(a_1))\alpha^{-1}_{i-j}(S^k(a_2))=(\alpha_{-j}*\alpha^{-1}_{i-j})(a),
\] for even $k$, and 
\begin{align*}
\alpha_i(S^{k+1}(a))&=\alpha_iS(S^k(a))=(\alpha_{j}*\alpha^{-1}_{i+j})(S^k(a))=\sum \alpha_{j}(S^k(a_2))\alpha^{-1}_{i+j}(S^k(a_1))\\
&=\sum (\alpha_{-j}*\alpha^{-1}_0)(a_2)(\alpha_i*\alpha^{-1}_{-j})(a_1)=\alpha_i(a),
\end{align*} for odd $k$.
Similarly, we can prove for $\alpha^{-1}_i(S^{k+1}(a))$.
\end{proof}

\begin{proposition}
\label{LiftLTS}
Let $B$ be a bialgebra satisfying the twisting conditions. Then any system of twisting functionals on $B$ can be extended uniquely to a system of twisting functionals on its Hopf envelope $\mathcal H(B)$. Moreover, any system of twisting functionals on $\mathcal H(B)$ is obtained from some system of twisting functionals on $B$ in such a way.
\end{proposition}
\begin{proof}
Let $\alpha=\{\alpha_i\}_{i\in \mathbb Z}$ be a system of twisting functionals on $B$.
We use the presentation of $\mathcal H(B)=T/W$ based on $B=\kk\langle V\rangle/(R)$ as discussed above. 

First, we lift the system of twisting functionals $\alpha=\{\alpha_i\}_{i\in \mathbb Z}$ to the free bialgebra $\kk\langle V\rangle$ in the following way. By formulas \eqref{eq:inductivelinear}, we can extend the restrictions $\alpha_i|_V$ and $\alpha^{-1}_i|_V$ on the subcoalgebra $V$ to the free bialgebra $\kk\langle V\rangle$. By abuse of notation, we still write them as  $\alpha=\{\alpha_i\}_{i\in \mathbb Z}$ and $\alpha^{-1}=\{\alpha^{-1}_i\}_{i\in \mathbb Z}$. It is routine to check that $\alpha$ is a system of twisting functionals on $\kk\langle V\rangle$ with convolution inverse $\alpha^{-1}$. Moreover, $\alpha_i(R)=\alpha^{-1}_i(R)=0$ which factor through $B=\kk \langle V\rangle/(R)$ giving back the original system of twisting functionals on $B$.

For simplicity, we write $V^{(k)}=S^k(V)$ in $T=\kk\langle \oplus_{k \ge 0} V^{(k)}\rangle$. We now extend $\alpha$ and $\alpha^{-1}$ from $\kk\langle V\rangle$ to $T$ by \eqref{eq:sigmaS}. Again, it is straightforward to check that $\alpha$ is a system of twisting functionals on $T$, with convolution inverse $\alpha^{-1}$,  extending that on $\kk\langle V\rangle$. By \Cref{twistL-deg-1}, it remains to show that  $\alpha(W)=0$, which would then yield a system of twisting functionals on $\mathcal H(B)=T/W$ extending that on $B$ via the natural bialgebra map $B\to \mathcal H(B)$. We will show that $\alpha$ and $\alpha^{-1}$ vanish on
\begin{gather*}
S^{\ell_1}(V)\otimes \cdots \otimes S^{\ell_p}(V)\otimes S^k(R)\otimes S^{\ell_{p+1}}(V)\otimes \cdots \otimes S^{\ell_{p+q}}(V),\\
S^{\ell_1}(V)\otimes \cdots \otimes S^{\ell_p}(V)\otimes (M \circ (\id \otimes S) \circ \Delta-u \circ \vps)(S^k(V))\otimes S^{\ell_{p+1}}(V)\otimes \cdots \otimes S^{\ell_{p+q}}(V), \\
S^{\ell_1}(V)\otimes \cdots \otimes S^{\ell_p}(V)\otimes (M \circ (S\otimes \id) \circ \Delta-u \circ \vps)(S^k(V))\otimes S^{\ell_{p+1}}(V)\otimes \cdots \otimes S^{\ell_{p+q}}(V),
\end{gather*}
by induction on $p+q$.

{\bf Case 1:} Assume $p+q=0$. By \eqref{eq:sigmaS}, we have for any homogeneous element $r\in R$:
\[
\alpha_i(S^k(r))=
\begin{cases}
\alpha_i(r) & \text{$k$ is even}\\
(\alpha_{-|r|}*\alpha^{-1}_{i-|r|})(r) & \text{$k$ is odd},\\
\end{cases}
\]
where $|r|$ denotes the degree of $r$.
Since $\alpha_i(R)=\alpha^{-1}_i(R)=0$ and $\Delta(R)\subseteq \kk\langle V\rangle \otimes (R)+(R)\otimes \kk \langle V\rangle$, one can check that $\alpha_i(S^k(R))=0$. A similar argument yields $\alpha^{-1}_i(S^k(R))=0$. Take any homogeneous element $a\in V$ of degree $j$, we have for $k$ even, 
\begin{align*}
\alpha_i \left((M \circ (\id \otimes S) \circ \Delta-u \circ \vps)(S^k(a)\right) &=\alpha_i \left (\sum S^k(a_1)S^{k+1}(a_2)-\vps(a) \right) \\ 
&=\sum \alpha_i(S^k(a_1))(\alpha_{i+j}*\alpha^{-1}_{j})(S^{k+1}(a_2))-\vps(a) \\ 
&=\sum \alpha_i(a_1)(\alpha_{j-j}*\alpha^{-1}_{i+j-j})(a_2)-\vps(a) \\ 
&=\sum \alpha_i(a_1)\alpha^{-1}_i(a_2)-\vps(a) =0,
\end{align*}
and for $k$ odd, 
\begin{align*}
\alpha_i\left((M \circ (\id \otimes S) \circ \Delta-u \circ \vps)(S^k(a)\right)&=\alpha_i\left (\sum S^{k}(a_2)S^{i+1}(a_1)-\vps(a) \right)\\
&=\sum \alpha_i(S^{k}(a_2))(\alpha_{i-j}*\alpha^{-1}_{-j})(S^{i+1}(a_1))-\vps(a)\\
&=\sum (\alpha_{-j}*\alpha^{-1}_{i-j})(a_2)(\alpha_{i-j}*\alpha^{-1}_{-j})(a_1)-\vps(a)\\
&=\sum \alpha_{-j}(a_2)\alpha^{-1}_{j}(a_1)-\vps(a) =0.
\end{align*}
Similarly, we can show that $\alpha_i\left((M \circ (S \otimes \id) \circ \Delta-u \circ \vps)(S^k(V)\right)=0$ for $k\ge 0$ and also for $\alpha^{-1}_i$. This completes the $p+q=0$ case. 

{\bf Case 2:} Suppose $p+q>0$. Set 
\[I_{p,q}=S^{\ell_1}(V)\otimes \cdots \otimes S^{\ell_p}(V)\otimes (M \circ (\id \otimes S) \circ \Delta-u \circ \vps)(S^k(V))\otimes S^{\ell_{p+1}}(V)\otimes \cdots \otimes S^{\ell_{p+q}}(V).\] 
We first claim that $I_{p,q}$ is a co-ideal in $T$, that is, $\Delta(I_{p,q})\subseteq T\otimes I_{p,q}+I_{p,q}\otimes T$. If $k$ is even, we have 
\begin{align*}
  \Delta(  (M \circ (\id \otimes S) \circ \Delta-u \circ \vps)S^k(a))=\sum S^k(a_1)S^{k+1}(a_3)\otimes (M \circ (\id \otimes S) \circ \Delta-u \circ \vps)S^k(a_2)\subseteq T\otimes I_{0,0}.
\end{align*}
Then it is direct to check that $\Delta(I_{p,q})\subseteq T\otimes T_{p,q}$. The argument for $k$ is odd is the same. This proves our claim. Now let $p>0$. For any $a\in I_{p,q}$, without loss of generality, we can write $a=bc$ for some $b\in S^{\ell_1}(V), c\in I_{p-1,q}$. So, we can apply Lemma \ref{Esigma}(2) to obtain that 
$\alpha_i(bc)=\sum \alpha_i(b)\alpha_{i+|b|}(c_1)\alpha_{|b|}^{-1}(c_2)=0$ since either $c_1$ or $c_2\in I_{p-1,q}$. The case for $q>0$ and $\alpha_i^{-1}$ can be argued analogously. Hence $\alpha_i(I_{p,q})=\alpha^{-1}_i(I_{p,q})=0$.
By the same argument, we can show for
\[ J_{p,q}=S^{\ell_1}(V)\otimes \cdots \otimes S^{\ell_p}(V)\otimes (M \circ (S\otimes \id) \circ \Delta-u \circ \vps)(S^k(V))\otimes S^{\ell_{p+1}}(V)\otimes \cdots \otimes S^{\ell_{p+q}}(V).\] 
This concludes the inductive step. Finally, the uniqueness of the extension of $\alpha$ from $B$ to $\mathcal H(B)$ follows from \Cref{lem:inverseSLinear}. 
\end{proof}

%%%%%%%%%%%%%%%%%%%%%%%%%%%%%%%%%%%
%%%%%%%%%%%%%%%%%%%%%%%%%%%%%%%%%%%
\section{2-cocycles via twisting system pairs }
\label{sect:2coc}

Throughout this section, let $B$ be a bialgebra satisfying the twisting conditions given in \Cref{defn:twisting conditions}. In this section, we introduce the notion of a twisting system pair of $B$, which we lift to that of its Hopf envelope $\mathcal H(B)$ and we use it to construct a certain 2-cocycle explicitly.  

\begin{lemma}\label{LTS}
Let $B$ be a bialgebra satisfying the twisting conditions. 
Consider a collection of linear functionals $\alpha=\{\alpha_i: i\in \mathbb Z\}$ with convolution inverse $\{\alpha_i^{-1}: i\in \mathbb Z\}$ on $B$. The following are equivalent: 
\begin{itemize}
    \item[(1)] The collection of maps $\alpha$ is a system of twisting functionals on $B$.
    \item[(2)] The collection of maps $\tau=\{\tau_i: i\in \mathbb Z\}$ with $\tau_i=\Xi^r[\alpha_i]$ is a twisting system of $B$. In this case, the inverse twisting system $\tau^{-1}$ is given by $\tau^{-1}_i=\Xi^r[\alpha_i^{-1}]$.
    \item[(3)] The collection of maps $\tau=\{\tau_i: i\in \mathbb Z\}$ with $\tau_i=\Xi^l[\alpha_i^{-1}]$  is a twisting system of $B$. In this case, the inverse twisting system $\tau^{-1}$ is given by $\tau^{-1}_i=\Xi^l[\alpha_i]$.
       \end{itemize}
\end{lemma}
\begin{proof}
(1)$\Rightarrow$(2): It is clear that for any $i \in \mathbb Z$,  $\tau_i(a)=\Xi^r[\alpha_i](a)=\sum a_1\alpha_i(a_2)$ is a graded linear automorphism of $B$ with inverse $\tau_i^{-1}(a)=\Xi^r[\alpha^{-1}_i](a)=\sum a_1 \alpha^{-1}_i(a_2)$, for any homogeneous $a\in B$ of degree $j$. Furthermore, we can compute that for $b \in B$ of any degree:
\begin{align*}
\tau_i(a\tau_j(b))=\tau_i\left (a \left (\sum b_1\alpha_j(b_2)\right ) \right )=\sum a_1b_1\alpha_i(a_2b_2)\alpha_j(b_3)=\sum a_1b_1\alpha_i(a_2)\alpha_{i+j}(b_2)=\tau_i(a)\tau_{i+j}(b).
\end{align*}
Moreover, we have $\tau_i(1)=1\alpha_i(1)=1$ and $\tau_0(a)=\sum a_1\alpha_0(a_2)=\sum a_1\vps(a_2)=a$. So $\tau=\{\tau_i:i\in \mathbb Z\}$ is a twisting system of $B$.

(2)$\Rightarrow$(1): Suppose $\tau=\{\tau_i: i\in \mathbb Z\}$ is a twisting system of $B$. Then, we can compute that 
\begin{align*}
\sum \alpha_i(ab_1)\alpha_j(b_2)&=\vps\left(\sum a_1b_1\alpha_i(a_2b_2)\alpha_j(b_3)\right) =\vps\left(\tau_i \left (a \left (\sum b_1\alpha_j(b_2)\right ) \right )\right)=\vps\left(\tau_i(a\tau_j(b))\right)\\
&=\vps\left(\tau_i(a)\tau_{i+j}(b)\right) =\vps\left(\sum a_1\alpha_i(a_2)b_1\alpha_{i+j}(b_2)\right) =\alpha_i(a)\alpha_{i+j}(b).
\end{align*}
Also we have $\alpha_i(1)=1\alpha_i(1)=\tau_i(1)=1$ and $\alpha_0(a)=\vps(a_1\alpha_0(a_2))=\vps(\tau_0(a))=\vps(a)$. Moreover, let $\beta_i=\vps \circ
\tau_i^{-1}$. Since $\Delta \circ \tau_i=(\id \otimes \tau_i)\circ \Delta$, one has $\Delta \circ \tau_i^{-1}=(\id \otimes \tau_i^{-1}) \circ \Delta$. Then one can check that $\tau_i^{-1}=\Xi^r[\beta_i]$. Hence $\tau_i\circ \tau_i^{-1}=\Xi^r[\alpha_i*\beta_i]=\id _B$ and $\tau_i^{-1}\circ \tau_i=\Xi^r[\beta_i*\alpha_i]=\id_B$, and so $\alpha_i*\beta_i=\beta_i*\alpha_i=u \circ \vps$ and $\beta_i=\alpha^{-1}_i$. Thus, $\alpha=\{\alpha_i\}_{i\in \mathbb Z}$ is a system of twisting functionals of on $B$.

(1)$\Leftrightarrow$(3): This can be proved in a similar way.
\end{proof}

\begin{defn}
[Twisting system pair]
\label{defn:twisting pair}
Let $(B,M,u,\Delta,\vps)$ be a bialgebra satisfying the twisting conditions. A pair $(\tau,\mu)$ of twisting systems of $B$ is said to be a \emph{twisting system pair} if for all $i \in \mathbb Z$:
\begin{enumerate}
\item[(\textbf{P1})] $\Delta\circ \tau_i=(\id \otimes \tau_i)\circ \Delta$ and $\Delta\circ \mu_i=(\mu_i\otimes \id)\circ \Delta$, and  
 \item[(\textbf{P2})] $\vps \circ (\tau_i \circ \mu_i)=\vps$.
\end{enumerate} 
\end{defn}

Using an argument similar to \cite[Lemma 2.1.2]{HNUVVW21}, we show in \Cref{lem:winding} that for any twisting system pair $(\tau,\mu)$ of a bialgebra $B$, $\tau$ and $\mu$ are uniquely determined by each other as winding linear maps. 

\begin{lemma}
\label{lem:winding}
Let $B$ be a bialgebra satisfying the twisting conditions. For any twisting system pair $(\tau,\mu)$ of $B$, we have a system of twisting functionals $\alpha=\{\alpha_i:i\in \mathbb Z\}$ on $B$ such that $\tau=\{\tau_i=\Xi^r(\alpha_i):i\in \mathbb Z\}$ and $\mu=\{\mu_i=\Xi^l(\alpha_i^{-1}): i\in \mathbb Z\}$. Moreover, for any $i,j \in \mathbb Z$, we have the following properties:
\begin{enumerate}
    \item[(\textbf{P3})] $\tau_i\circ \mu_j=\mu_j\circ \tau_i$, and
    \item[(\textbf{P4})] $(\tau_i\otimes \mu_i)\circ \Delta=\Delta$.
\end{enumerate}
\end{lemma}
\begin{proof}
 Let $\alpha_i=\vps \circ \tau_i$ and $\alpha^{-1}_i=\vps \circ \tau_i^{-1}$. Then we have
\[
\tau_i(a)=\sum \tau_i(a)_1\vps(\tau_i(a)_2) \overset{(\textbf{P1})} = \sum a_1 \vps(\tau_i(a_2))=\Xi^r[\vps \circ \tau_i](a)=\Xi^r[\alpha_i](a).
\] 
Hence $\tau_i=\Xi^r[\alpha_i]$ and by Lemma \ref{LTS}, $\alpha=\{\alpha_i: i\in \mathbb Z\}$ is a system of 
twisting functionals on $B$. Since $\tau^{-1}$ satisfies (\textbf{P1}), we have $\tau_i^{-1}=\Xi^r[\alpha^{-1}_i]$. A straightforward computation shows that $\tau_i\circ \tau_i^{-1}=\Xi^r[\alpha_i*\alpha^{-1}_i]=\id_B$ and $\tau_i^{-1}\circ \tau_i=\Xi^r[\alpha^{-1}_i*\alpha_i]=\id_B$. This 
implies that $\alpha_i$ and $\alpha^{-1}_i$ are convolution inverse of each other. Similarly, we can show that $\mu_i=\Xi^l[\vps \circ \mu_i]$. Condition (\textbf{P2}) implies that $\vps = \vps \circ (\tau_i\circ \mu_i)=(\vps \circ \mu_i) *\alpha_i$. Hence we have $\vps \circ \mu_i=\alpha^{-1}_i$ and $\mu_i=\Xi^l[\alpha^{-1}_i]$. 
Finally, for any $i,j \in \mathbb Z$, condition (\textbf{P3}) holds since
\begin{align*}
(\tau_i\circ \mu_j)(a)&=\tau_i\left(\Xi^l[\alpha^{-1}_j](a)\right) =\Xi^r[\alpha_i]\left(\sum \alpha^{-1}_j(a_1)a_2\right) =\sum \alpha^{-1}_j(a_1)a_2\alpha_i(a_3)\\
&=\Xi^l[\alpha^{-1}_j]\left(\sum a_1\alpha_i(a_2)\right) =\Xi^l[\alpha^{-1}_j]\left(\Xi^r[\alpha_i](a)\right) =(\mu_j\circ \tau_i)(a),
\end{align*}
and condition (\textbf{P4}) holds since
\[
(\tau_i\otimes \mu_i)\Delta(a) =\sum \Xi^r[\alpha_i](a_1)\otimes \Xi^l[\alpha^{-1}_i(a_2)] =\sum a_1\alpha_i(a_2)\otimes \alpha^{-1}_i(a_3)a_4 =\sum a_1\otimes a_2 =\Delta(a).
\]
\end{proof}

\begin{Cor}
\label{twist-pair-envelope-ext}
   Let $B$ be a bialgebra satisfying the twisting conditions. Then any twisting system pair of a bialgebra $B$ can be extended uniquely to a twisting system pair of its Hopf envelope $\mathcal H(B)$. Moreover, any twisting system pair of $\mathcal H(B)$ is obtained from some twisting system pair of $B$ in such a way.
\end{Cor}

\begin{proof}
This is a direct consequence of Lemma \ref{lem:winding} and \Cref{LiftLTS}.
\end{proof}

Now, we consider any Hopf algebra $H$ satisfying the twisting conditions. A \emph{right 2-cocycle} on $H$ is a convolution invertible linear map $\sigma: H\otimes H\to \kk$ satisfying
\begin{equation}\label{eq:2cocycle}
\sum \sigma(x_1y_1, z)\sigma(x_2, y_2)=\sum \sigma(x, y_1z_1)\sigma(y_2, z_2)\qquad \text{and}\qquad  \sigma(x,1)=\sigma(1,x)=\vps(x),
\end{equation} 
for all $x,y,z\in H$. The convolution inverse of $\sigma$, denoted by $\sigma^{-1}$, is a \emph{left 2-cocycle} on $H$. Given a right 2-cocycle $\sigma$, let $H^\sigma$ denote the coalgebra $H$ endowed with the original unit and deformed product
\[
x \cdot_\sigma y:=\sum \sigma^{-1}(x_1,y_1)\,x_2y_2\,\sigma(x_3,y_3),
\]
for any $x,y\in H$. In fact, $H^\sigma$ is a Hopf algebra with the deformed antipode $S^\sigma$ given in \cite[Theorem 1.6]{Doi93}. We call $H^\sigma$ the \emph{2-cocycle twist} of $H$ by $\sigma$. There is a monoidal equivalence 
\[
  F: \comod(H)~\overset{\otimes}{\cong}~ \comod(H^\sigma) \quad \text{ sending } \quad 
  U~\mapsto~F(U)=:U_\sigma.
\]
 We write $\otimes$ and $\otimes_\sigma$ for the tensor products in the corresponding right comodule categories. As a functor, $F$ is the identity functor since $H = H^\sigma$ as coalgebras. As a monoidal equivalence, $F$ is equipped with natural isomorphisms of $H^\sigma$-comodules: 
\begin{align*} 
\xi_{U,V}: F(U \otimes V )~&\xrightarrow{\sim}~ F(U) \otimes_\sigma F(V)\\
u\otimes v~&\mapsto~ \sum \sigma^{-1}(u_1,v_1)\,u_0\otimes v_0,\notag
\end{align*} 
compatible with the associativity, where the right coaction of $H$ on $U$ is given by $\rho: u \mapsto \sum u_0 \otimes u_1 \in U \otimes H$. In particular, $F$ sends a (connected graded) $H$-comodule algebra $A$ to the twisted (connected graded) $H^\sigma$-comodule algebra $F(A)=A_\sigma=A$ as vector spaces, with 2-cocycle twist multiplication $a \cdot_\sigma b=\sum a_0b_0\sigma(a_1,b_1)$, for any $a,b\in A$.

\begin{proposition}
\label{2-cocycle twist}
Let $H$ be a Hopf algebra satisfying the twisting conditions, and $(\tau,\mu)$ be a twisting system pair of $H$. Then $\tau\circ \mu=\{\tau_i\circ \mu_i: i\in \mathbb Z\}$ is a twisting system of $H$. Moreover, $H^{\tau\circ \mu}\cong H^\sigma$ as graded algebras, where the right 2-cocycle $\sigma: H \otimes H\to \kk$ and its convolution inverse $\sigma^{-1}$ are given by
\[
\sigma(x,y)=\vps(x)\vps(\tau_{|x|}(y)) \qquad \text{ and } \qquad \sigma^{-1}(x,y)=\vps(x)\vps(\mu_{|x|}(y))
\]
for any homogeneous elements $x,y\in H$ where $|x|, |y|$ denote the degrees of $x$ and $y$, respectively.
\end{proposition}

\begin{proof}
We first show that $\tau\circ \mu$ is a twisting system. It is clear that $\nu:=\{\nu_i=\tau_i\circ \mu_i: i\in \mathbb Z\}$ is a set of graded linear automorphisms with inverse $\nu^{-1}:=\{\nu_i^{-1}=\mu_i^{-1}\circ \tau_i^{-1}: i\in \mathbb Z\}$ on $H$.
By \Cref{lem:winding}, we have 
\begin{align}
\label{eq:twistingpair}
\tau_i=\Xi^r(\alpha_i), \qquad \tau_i^{-1}=\Xi^r(\alpha^{-1}_i), \qquad \mu_i=\Xi^l(\alpha^{-1}_i), \qquad \mu_i^{-1}=\Xi^l(\alpha_i),
\end{align}
for the system of twisting functionals $\alpha:=\{\alpha_i=\vps \circ \tau_i: i\in \mathbb Z\}$ on $H$. Let $x, y$ and $z$ be homogeneous elements in $H$. For any $i \in \mathbb Z$, $\nu$ is a twisting system of $H$ since 
\begin{align*}
 \nu_i(xy)&=\tau_i\circ\mu_i(xy)
 =\left(\tau_i\circ\mu_i(x)\right)\, \left(\tau_{i+|x|}\circ\tau_{|x|}^{-1}\circ \mu_{i+|x|}\circ\mu_{|x|}^{-1}(y)\right)\\
 \overset{(\textbf{P3})}&{=}\left(\tau_i\circ\mu_i(x)\right)\left(\tau_{i+|x|}\circ\mu_{i+|x|}\circ \mu_{|x|}^{-1}\circ \tau_{|x|}^{-1}(y)\right) 
 =\left(\nu_i(x)\right)\left(\nu_{i+|x|}\circ \nu_{|x|}^{-1}(y)\right).
\end{align*}
We show next that $\sigma$ satisfies \eqref{eq:2cocycle}: 
\begin{align*}
\sigma(x_1y_1,z)\sigma(x_2,y_2) &=\sum \vps(x_1y_1)\alpha_{|x|+|y|}(z)\vps(x_2)\alpha_{|x|}(y_2) 
=\sum \vps(x)\alpha_{|x|}(y)\alpha_{|x|+|y|}(z)\\
&=\sum \vps(x)\alpha_{|x|}(yz_1)\alpha_{|y|}(z_2) 
=\sum \vps(x)\alpha_{|x|}(y_1z_1)\vps(y_2)\alpha_{|y|}(z_2) 
=\sum \sigma(x,y_1z_1)\sigma(y_2,z_2),
\end{align*} where the third equality follows from \Cref{Esigma}(1) and
\(
\sigma(x,1)=\vps(x)\alpha_{|x|}(1)=\vps(x)=\vps(1)\alpha_0(x)=\sigma(1,x).
\) 
Note that it is straightforward to check that $\sigma$ is convolution invertible with inverse $\sigma^{-1}(x,y)=\vps(x)\alpha^{-1}_{|x|}(y)=\vps(x)\vps(\mu_{|x|}(y))$. Thus, $\sigma$ is a right 2-cocycle on $H$.  

We now show that $H^{\tau\circ \mu}\cong H^\sigma$ as graded algebras via the identity map on vector spaces. By \eqref{eq:twistingpair} and \Cref{lem:winding}, we indeed have
\begin{align*}
x \cdot_\sigma y &=\sum\sigma^{-1}(x_1, y_1)x_2y_2\sigma(x_3, y_3)
=\sum \vps(x_1)\alpha^{-1}_{|x|}(y_1)x_2y_2\vps(x_3)\alpha_{|x|}(y_3) \\
&=\sum x \alpha^{-1}_{|x|}(y_1)y_2\alpha_{|x|}(y_3)
=x\mu_{|x|}\tau_{|x|}(y)
=x \nu_{|x|}(y) = x \cdot_\nu y.
\end{align*}
Since $H^\sigma$ is a Hopf algebra, it implies that $H^{\tau\circ \mu}$ also has a Hopf algebra structure via the above identity isomorphism $\id: H^{\tau\circ \mu}\cong H^\sigma$. 
\end{proof}

\begin{proposition}
    \label{twist-envelope}
    Let $B$ be a bialgebra satisfying the twisting conditions, $(\tau, \mu)$ be a twisting system pair of $B$, and $(\mc{H}(\tau), \mc{H}(\mu))$ be the induced twisting system pair of $\mc{H}(B)$ via \Cref{twist-pair-envelope-ext}. Then $\mc{H}(B^{\tau \circ \mu}) \cong \mc{H}(B)^{\mc{H}(\tau) \circ \mc{H}(\mu)}$ as Hopf algebras. 
\end{proposition}

\begin{proof} 
Denote by $\tau^{-1}$ and $\mu^{-1}$ the inverse twisting systems of $\tau$ and $\mu$, respectively. Since $\tau^{-1}$ and $\mu^{-1}$ are twisting systems of $B^\tau$ and $B^\mu$ respectively, one can directly check that $(\tau^{-1},\mu^{-1})$ is the twisting system pair of $B^{\tau \circ \mu}$ such that $B\cong (B^{\tau \circ \mu})^{\tau^{-1}\circ \mu^{-1}}$ as bialgebras. Similarly, we write $(\mc{H}(\tau)^{-1},\mc{H}(\mu)^{-1})=(\mc{H}(\tau^{-1}),\mc{H}(\mu^{-1}))$ as the unique extension of the twisting system pair $(\tau^{-1},\mu^{-1})$ from $B^{\tau \circ \mu}$ to $\mc{H}(B^{\tau \circ \mu})$. 

We denote by $\iota_B: B\to \mc{H}(B)$ and $\iota_{B^{\tau \circ \mu}}: B^{\tau \circ \mu}\to \mc{H}(B^{\tau \circ \mu})$ the corresponding bialgebra maps from bialgebras to their Hopf envelopes satisfying the required universal property. 

By the universal property of the Hopf envelope, one has a unique Hopf algebra map $g: \mc{H}(B^{\tau \circ \mu})\to \mc{H}(B)^{\mc{H}(\tau)\circ\mc{H}(\mu)}$ where the following diagram commutes:
\[
\xymatrix{
B^{\tau \circ \mu}\ar[r]^-{\iota_{B^{\tau \circ \mu}}}\ar[rd]_-{(\iota_B)^{\tau\circ \mu}}& \mc{H}(B^{\tau \circ \mu})\ar[d]^-{g}\\
& \mc{H}(B)^{\mc{H}(\tau)\circ\mc{H}(\mu)}.
}
\]
Similarly, one has a unique Hopf algebra map $h: \mc{H}(B)^{\mc{H}(\tau)\circ\mc{H}(\mu)}\to \mc{H}(B^{\tau \circ \mu})$ making the diagram 
\[
\xymatrix{
B\ar[r]^-{\iota_B}\ar[rd]_-{(\iota_{B^{\tau\circ \mu}})^{\mc{H}(\tau)^{-1}\circ \mc{H}(\mu)^{-1}} \qquad} &\mc{H}(B)\ar[d]^-{h}\\
& \mc{H}(B^{\tau \circ \mu})^{\mc{H}(\tau)^{-1}\circ \mc{H}(\mu)^{-1}}
}
\]
commute. By letting $l=h^{\tau\circ \mu}$, we have the following commutative diagram: 
\[
\xymatrix{
B^{\tau\circ \mu}\ar[r]^-{(\iota_B)^{\tau\circ \mu}}\ar[rd]_-{(\iota_{B^{\tau\circ \mu}})} &\mc{H}(B)^{\mc{H}(\tau)\circ\mc{H}(\mu)}\ar[d]^-{l}\\
& \mc{H}(B^{\tau \circ \mu}).
}
\]
By the universal property of $\iota_B$ and $\iota_{B^{\tau\circ \mu}}$ again, one can show that $g\circ l$ and $l\circ g$ are identities on $\mc{H}(B)^{\mc{H}(\tau)\circ\mc{H}(\mu)}$ and $\mc{H}(B^{\tau \circ \mu})
$, respectively. This completes our proof.  
\end{proof}

%%%%%%%%%%%%%%%%%%%%%%%%%%%%%%%%%%%
%%%%%%%%%%%%%%%%%%%%%%%%%%%%%%%%%%%
\section{Proof of \Cref{atau-qsequiv}}
\label{sect:qsequiv}

Throughout this section, let $A$ be an $m$-homogeneous algebra and $A^!$ be its Koszul dual. Let $\tau$ be a twisting system of $A$ and $\tau^!$ be the dual twisting system of $A^!$, defined in \Cref{sect:genresults}. In the following results, we find a twisting pair of $A \bullet A^! \cong \uend^r(A)$ and lift it to give a Hopf algebra isomorphism between the universal quantum algebra of the Zhang twist $A^\tau$ and the 2-cocycle twist of the universal quantum algebra of $A$ (see \Cref{univ-a-twist}). We then prove our main result, \Cref{atau-qsequiv}, which states that if two $m$-homogeneous algebras are graded Morita equivalent then they are quantum-symmetrically equivalent.

\begin{lemma}
    If $A$ is an $m$-homogeneous algebra with twisting system $\tau$, then $\tau \bullet \id$ and $\id \bullet \tau^!$ (defined in \Cref{sect:genresults}) form a twisting system pair of $A \bullet A^! \cong \uend^r(A)$. Moreover, we have the  commutative diagrams:
\begin{align}\label{lem:cend}
\xymatrix{
A\ar[r]^-{\rho}\ar[d]_-{\tau} & A\otimes \underline{\rm end}^r(A)\ar[d]^-{\id \otimes (\tau\bullet \id)}\\
A\ar[r]^-{\rho} & A\otimes \underline{\rm end}^r(A)
}
\qquad \text{and} \qquad
\xymatrix{
A^!\ar[r]^-{\rho^!}\ar[d]_-{\tau^!} &  \underline{\rm end}^r(A)\otimes A^!\ar[d]^-{(\id\bullet \tau^!)\otimes \id}\\
A^!\ar[r]^-{\rho^!} &  \underline{\rm end}^r(A)\otimes A^!.
}
\end{align} 
\end{lemma}

\begin{proof}
    We know that both $\mu:=\tau \bullet \id$ and $\xi:=\id \bullet \tau^!$ are twisting systems of $\underline{\rm end}^r(A)=A \bullet A^!$, by \Cref{prop:twistingtau!} and \Cref{bullet-twist}. Suppose $\{x_1,..., x_n\}$ is a basis of $A_1$; denote the dual basis of $A_1^!$ by $\{x^1,..., x^n\}$. Recall that the coaction of $A \bullet A^!$ on $A$ sends
    \[
    \rho: x_j \mapsto \sum x_k \otimes z_j^k,
    \]
    where $z_j^k$ is the image of $x_j \otimes x^k$ in $A \bullet A^!$. Since each linear automorphism $\tau_i$ preserves degrees, we have some invertible scalar matrix $(\lambda^i_{jl})_{1\leq j,l\leq n}$ with inverse $(\phi^i_{jl})_{1\leq j,l\leq n}$ such that
   \begin{align}\label{eq:pairE}
   \tau_i : x_j \mapsto \sum \lambda^i_{jl} x_l, \qquad  \mu_i(z^k_j)=\sum_{1\leq l\leq n} \lambda^i_{jl}z_l^k, \qquad \text{and} \qquad   \xi_i(z^k_j)=\sum_{1\leq l\leq n} z_j^l\phi^i_{lk}.
    \end{align}
We show that (\textbf{P1}) and (\textbf{P2}) hold for $\mu$ and $\xi$ by induction on the degrees in $A\bullet A^!$. It is trivial for degree $0$ and straightforward for degree $1$ due to \eqref{eq:pairE}. Suppose (\textbf{P1}) and (\textbf{P2}) hold for all degrees $\leq n$. Take any homogeneous elements $a,b$ in $A\bullet A^!$ with $a$ of degree $j$ and $b$ of degree $n+1-j$. Then one can check that 
\begin{align*}
\Delta \circ \mu_i(ab) &= \Delta \circ (\mu_i(a)\mu_{i+j}\mu_j^{-1}(b)) = (\Delta \circ \mu_i)(a) (\Delta \circ \mu_{i+j}\mu_j^{-1}) (b) \\
&= (\id \otimes \mu_i) \circ \Delta(a) (\id \otimes \mu_{i+j}\mu_j^{-1}) \circ \Delta(b) = \sum a_1b_1 \otimes \mu_i(a_2)\mu_{i+j}\mu_j^{-1}(b_2) \\
&= \sum a_1b_1 \otimes \mu_i(a_2b_2) = (\id \otimes \mu_i) \circ \Delta(ab).
\end{align*}
So (\textbf{P1}) holds for $\mu$ and similarly for $\xi$. Now for (\textbf{P2}), we have
\begin{align*}
\vps \circ (\mu_i\circ \xi_i)(ab)&=\vps\mu_i\xi_i(a)\vps (\mu_{i+j}\mu_j^{-1}\xi_{i+j}\xi^{-1}_j)(b) =\vps\mu_i\xi_i(a)\vps (\mu_{i+j}\xi_{i+j}\mu_j^{-1}\xi^{-1}_j)(b) \\
&=\vps(a)\vps(\mu_j^{-1}\xi^{-1}_j)(b)
=\vps(a)\vps(b) =\vps(ab).
\end{align*}
Hence $(\mu,\xi)$ is a twisting system pair of $\underline{\rm end}^r(A)$.
  
For the diagrams in \eqref{lem:cend}, we will show the first diagram is commutative. A similar argument can be applied to show the second diagram is commutative. One can check that 
\begin{align*}
    (\id_A \otimes \mu_i) \rho (x_j)=  (\id_A \otimes \mu_i) \left (\sum x_k \otimes z_j^k \right )= \sum_{k,l} x_k  \otimes \lambda_{jl}^i z_l^k=\rho\left ( \sum_{l} \lambda_{jl}^i x_l\right )=  \rho \tau_i(x_j).
\end{align*}
Note that by a similar argument, we also have $(\id_A \otimes \mu_i^{-1}) \rho = \rho \tau_i^{-1}$.
Now by an inductive argument, we prove that the diagram commutes in degree $n$, supposing that for any degree $n-1$ element $a$, we have
\[
    \rho \tau_i (a) =(\id_A \otimes \mu_i) \rho\quad \text{and}\quad   \rho \tau_i^{-1} (a) =(\id_A \otimes \mu_i^{-1}) \rho.
\]
Of course, it is enough to check on degree $n$ elements of the form $x a$, where $x \in A_1$ and $a \in A_{n-1}$, since we are assuming $A$ is generated in degree 1. Now we can check 
\begin{align*}
    \rho \tau_i(xa)&= \rho(\tau_i(x) \tau_{i+1} \tau_1^{-1} (a))=\rho \tau_i(x) \rho \tau_{i+1} \tau^{-1}_1(a) \\
    &=(\id_A \otimes\mu_i)\rho(x) (\id_A \otimes \mu_{i+1}) (\id_A \otimes \mu^{-1}_1) \rho(a) =(\id_A \otimes \mu_i) \rho (xa).
\end{align*}
The argument for $\rho \tau_i^{-1}$ is similar. By induction, the diagram commutes in all degrees. 
\end{proof}

Let $\sigma$ be the 2-cocycle of $\uaut^r(A)$ corresponding to the twisting system pair $(\mathcal{H}(\tau \bullet \id), \mathcal{H}(\id \bullet \tau^!))$ in \Cref{twist-pair-envelope-ext} and \Cref{2-cocycle twist}. We know that $\uaut^r(A)^{\mathcal{H}(\tau \bullet \id)\circ \mathcal{H}(\id \bullet \tau)} \cong \uaut^r(A)^\sigma$.
Moreover, by the universal property of the Hopf envelope and \eqref{lem:cend}, the diagrams 
\begin{align}\label{lem:caut}
\xymatrix{
A\ar[r]^-{\rho}\ar[d]_-{\tau} & A\otimes \underline{\rm aut}^r(A)\ar[d]^-{\id \otimes {\mc H}(\tau\bullet \id)}\\
A\ar[r]^-{\rho} & A\otimes \underline{\rm aut}^r(A)}
\qquad \text{and} \qquad 
\xymatrix{
A^!\ar[r]^-{\rho^!}\ar[d]_-{\tau^!} &  \underline{\rm aut}^r(A)\otimes A^!\ar[d]^-{{\mc H}(\id\bullet \tau^!)\otimes \id}\\
A^!\ar[r]^-{\rho^!} &  \underline{\rm aut}^r(A)\otimes A^!
}
\end{align}
commute. We use the next result to prove that quantum-symmetric equivalence is a graded Morita invariant.

\begin{lemma}
\label{univ-a-twist}
Let $A$ be an $m$-homogeneous algebra and $\tau$ be a twisting system of $A$. We have an isomorphism of Hopf algebras $\uaut^r(A^{\tau}) \cong \uaut^r(A)^{\sigma}$, where $\sigma$ is the right 2-cocycle corresponding to the twisting system pair $(\mc{H}(\tau \bullet \id), \mc{H}( \id \bullet \tau^!))$ defined in \Cref{2-cocycle twist}. 
\end{lemma}
\begin{proof}
We check that 
\begin{align*}
    \uend^r(A^\tau) &\cong A^\tau \bullet (A^\tau)^! 
    \cong A^\tau \bullet (A^!)^{\tau^!}\cong (A \bullet A^!)^{\tau \bullet \tau^!}\cong \uend^r(A)^{\tau \bullet \tau^!},
\end{align*}
where the second isomorphism follows by \Cref{prop:tau!}, and the third isomorphism follows from \Cref{bullet-twist}. 
Then we can show that
\begin{align*}
    \uaut^r(A^\tau) \cong \mathcal{H}(\uend^r(A^\tau))
    &\cong \mathcal{H}(\uend^r(A)^{\tau \bullet \tau^!})\cong \mathcal{H}(\uend^r(A))^{\mathcal{H}(\tau \bullet \tau^!)}\cong \uaut^r(A)^\sigma,
\end{align*}
where the second isomorphism follows from our above computation, the third isomorphism follows from \Cref{twist-envelope}, and the fourth isomorphism follows from \Cref{2-cocycle twist}.
\end{proof}

\noindent 
{\bf Proof of \Cref{atau-qsequiv}.} 
Let $A$ and $B$ be two $m$-homogeneous algebras that are graded Morita equivalent; we must show that they are quantum-symmetrically equivalent.
 Without loss of generality, by \cite[Theorem 1.2]{Zhang1996}, we can assume $B=A^\tau$ for some twisting system $\tau=\{\tau_i: i\in \mathbb Z\}$ of $A$. By \Cref{univ-a-twist}, 
there exists a right 2-cocycle $\sigma$ on $\uaut^r(A)$ given by the twisting system pair $(\mc{H}(\tau \bullet \id), \mc{H}(\id \bullet \tau^!))$ such that $\uaut^r(A^{\tau}) \cong \uaut^r(A)^{\sigma}$. As a consequence, $\comod(\uaut^r(A))$ and $\comod(\uaut^r(A^{\tau}))$ are monoidally equivalent. Since $A$ is an $\uaut^r(A)$-comodule algebra, we can consider the corresponding $\uaut^r(A)^\sigma$-comodule algebra $A_{\sigma}$. It remains to show that there is an isomorphism $A_{\sigma} \cong A^\tau$ of $\uaut^r(A)^\sigma$-comodule algebras. The following computation concludes the proof: for any homogeneous elements $a,b\in A$,
  \begin{align*}
        a \cdot_{\sigma} b & = \sum a_0 b_0 \sigma(a_{1},b_{1})=\sum a_0 b_0 \varepsilon(a_1) \varepsilon({\mc H}(\tau\bullet \id)_{|a|}(b_1))=a ((\varepsilon \otimes \id) \circ ( \id \otimes {\mc H}(\tau\bullet \id)) \circ \rho )(b)\\
        & =a (\varepsilon \otimes \id)(\rho(\tau_{|a|}(b)))=a \tau_{|a|}(b)=a \cdot_{\tau} b,
    \end{align*}
the fourth equality follows from \eqref{lem:caut}.
\hfill $\square$
\vspace{\baselineskip}

The following is now an immediate consequence of the main results proved in our paper. 

\begin{Cor}
    Let $A$ be any $m$-homogeneous algebra and $H$ a Hopf algebra that right coacts on $A$ by preserving its grading. Then for any right 2-cocycle $\sigma$ on $H$, the following are equivalent.
    \begin{itemize}
\item[(i)] The 2-cocycle twist algebra $A_\sigma$ and $A$ are graded Morita equivalent.
\item[(ii)] There is a twisting system $\tau$ on $A$ such that $A_\sigma \cong A^\tau$ as graded algebras. 
\item[(iii)] There is a 2-cocycle $\sigma'$ on $\uaut^r(A)$ given by some twisting system pair such that $A_\sigma \cong A_{\sigma'}$ as algebras. 
    \end{itemize}
\end{Cor}

\begin{proof}
(i)$\Leftrightarrow$(ii) Note that by \cite[Lemma 4.1.5]{HNUVVW21}, $A_\sigma$ is again an $m$-homogeneous algebra. So the equivalence directly follows from \cite[Theorem 1.2]{Zhang1996}.

(ii)$\Rightarrow$(iii) It is derived from the proof of \Cref{atau-qsequiv}, where the twisting system pair is given in \Cref{univ-a-twist}.

(iii)$\Rightarrow$(ii) Without loss of generality, we can assume the 2-cocycle $\sigma$ is given by some twisting system $(f,g)$ on $\uaut^r(A)$. By \Cref{lem:winding}, there is a system of twisting functionals $\{\alpha_i:i\in \mathbb Z\}$ on $\uaut^r(A)$ such that $f_i=\Xi^r[\alpha_i]$ and $g_i=\Xi^l[\alpha_i^{-1}]$. We define a collection of graded linear automorphisms $\tau=\{\tau_i: i\in \mathbb Z\}$ on $A$ via $\tau_i(a)=\sum a_0\alpha_i(a_1)$ with linear inverse $\tau^{-1}_i(a)=\sum a_0\alpha_i^{-1}(a_1)$. Similar to \Cref{LTS}, one can easily check that $\tau$ is a twisting system on $A$. Note the 2-cocycle $\sigma$ on $\uaut^r(A)$ is given by $\sigma(x,y)=\vps(x)\alpha_{|x|}(y)$ for any homogeneous elements $x,y\in \uaut^r(A)$. Therefore, we have
\begin{align*}
a\cdot_\sigma b=\sum a_0b_0\sigma(a_1,b_1)=\sum ab_0\alpha_{|a|}(b_1)=a\tau_{|a|}(b)=a\cdot_\tau b
\end{align*}
for any homogeneous elements $a,b\in A$. This proves the implication.
\end{proof}

\begin{remark}
In \cite{ArtinZhang1994}, Artin and Zhang introduced the concept of a noncommutative projective scheme ${\rm Proj}(A)$, which gives an analogue of the category of quasi-coherent sheaves for the noncommutative projective space associated to $A$. Since ${\rm Proj}(A)$ is a quotient of $\grmod(A)$, and we have proven that $QS(A)$ only depends on $\grmod(A)$, one might ask whether $QS(A)$ is actually an invariant of ${\rm Proj}(A)$. However, we point out that there are connected graded algebras whose noncommutative projective schemes are equivalent but are not quantum-symmetrically equivalent. For example, let $A$ be a polynomial algebra and $B=A^{\langle d\rangle}$ be the Veronese subalgebra, which always shares the same ${\rm Proj}$ with $A$ (see e.g., \cite[Introduction]{MU} for further details on the Veronese subalgebra). By \cite[Lemma 3.2.7]{HNUVVW3}, $A$ and $B$ are not quantum-symmetrically equivalent since $A$ has a finite global dimension, but $B$ does not when $d\ge 2$. 
\end{remark}

We speculate that \Cref{atau-qsequiv} holds in general for any two graded algebras that are finitely generated in degree one, without the $m$-homogeneous assumption. Since \Cref{atau-qsequiv} implies that the tensor category $\comod(\uaut^r(A))$ depends only on $\grmod(A)$ rather than on $A$, we ask the following question.

\begin{ques}
    For a connected graded algebra $A$ that is finitely generated in degree one, is there an intrinsic categorical construction for $\comod(\uaut^r(A))$ purely in terms of $\grmod(A)$?
\end{ques}

%%%%%%%%%%%%%%%%%%%%%%%%%%%%%%%%%%%
%%%%%%%%%%%%%%%%%%%%%%%%%%%%%%%%%%%

\bibliography{bibfile}

\providecommand{\bysame}{\leavevmode\hbox to3em{\hrulefill}\thinspace}
\providecommand{\MR}{\relax\ifhmode\unskip\space\fi MR }
\providecommand{\MRhref}[2]{%
  \href{http://www.ams.org/mathscinet-getitem?mr=#1}{#2}
}
\providecommand{\href}[2]{#2}
\begin{thebibliography}{10}

\bibitem{Agore2021}
A.~L. Agore, \emph{Universal coacting {P}oisson {H}opf algebras}, Manuscripta
  Math. \textbf{165} (2021), no.~1-2, 255--268. \MR{4242570}

\bibitem{AGV}
A.~L. Agore, A.~S. Gordienko, and J.~Vercruysse, \emph{{$V$}-universal {H}opf
  algebras (co)acting on {$\Omega$}-algebras}, Commun. Contemp. Math.
  \textbf{25} (2023), no.~1, Paper No. 2150095, 40. \MR{4523148}

\bibitem{ArtinZhang1994}
M.~Artin and J.~J. Zhang, \emph{Noncommutative projective schemes}, Adv. Math.
  \textbf{{\bf 109}} (1994), no.~2, 228--287. \MR{1304753}

\bibitem{BZ08}
K.~A. Brown and J.~J. Zhang, \emph{Dualising complexes and twisted {H}ochschild
  (co)homology for {N}oetherian {H}opf algebras}, J. Algebra \textbf{{\bf 320}}
  (2008), no.~5, 1814--1850. \MR{2437632}

\bibitem{Chirvasitu-Walton-Wang2019}
A.~Chirvasitu, C.~Walton, and X.~Wang, \emph{On quantum groups associated to a
  pair of preregular forms}, J. Noncommut. Geom. \textbf{{\bf 13}} (2019),
  no.~1, 115--159. \MR{3941475}

\bibitem{Doi93}
Y.~Doi, \emph{Braided bialgebras and quadratic bialgebras}, Comm. Algebra
  \textbf{{\bf 21}} (1993), no.~5, 1731--1749. \MR{1213985}

\bibitem{EGNO}
P.~Etingof, S.~Gelaki, D.~Nikshych, and V.~Ostrik, \emph{Tensor categories},
  Mathematical Surveys and Monographs, vol.~{\bf 205}, American Mathematical
  Society, Providence, RI, 2015. \MR{3242743}

\bibitem{HNUVVW21}
H.~Huang, V.~C. Nguyen, C.~Ure, K.~B. Vashaw, P.~Veerapen, and X.~Wang,
  \emph{Twisting of graded quantum groups and solutions to the quantum
  {Y}ang-{B}axter equation}, to appear in Transform. Groups (2022).

\bibitem{HNUVVW2}
\bysame, \emph{A cogroupoid associated to preregular forms}, to appear in J.
  Noncommut. Geom. (2024).

\bibitem{HNUVVW3}
\bysame, \emph{Twisting {M}anin's universal quantum groups and comodule
  algebras}, Adv. Math. \textbf{445} (2024), Paper No. 109651. \MR{4732072}

\bibitem{HWWW}
H.~Huang, C.~Walton, E.~Wicks, and R.~Won, \emph{Universal quantum
  semigroupoids}, J. Pure Appl. Algebra \textbf{{\bf 227}} (2023), no.~2, Paper
  No. 107193, 34. \MR{4460365}

\bibitem{LopezWalton}
F.~Liu~Lopez and C.~Walton, \emph{Twists of graded algebras in monoidal
  categories}, arXiv:2311.18105.

\bibitem{Manin2018}
Y.~I. Manin, \emph{Quantum groups and noncommutative geometry}, second ed., CRM
  Short Courses, Centre de Recherches Math\'{e}matiques, [Montreal], QC;
  Springer, Cham, 2018, With a contribution by Theo Raedschelders and Michel
  Van den Bergh. \MR{3839605}

\bibitem{Mori-Smith2016}
I.~Mori and S.~P. Smith, \emph{{$m$}-{K}oszul {A}rtin--{S}chelter regular
  algebras}, J. Algebra \textbf{{\bf 446}} (2016), 373--399. \MR{3421098}

\bibitem{MU}
I.~Mori and K.~Ueyama, \emph{A categorical characterization of quantum
  projective spaces}, J. Noncommut. Geom. \textbf{15} (2021), no.~2, 489--529.
  \MR{4325714}

\bibitem{Pareigis}
B.~Pareigis, \emph{{Q}uantum {G}roups and {N}oncommutative {G}eometry}, Lecture
  Notes TU Munich.

\bibitem{Porst2012}
H.-E. Porst, \emph{Takeuchi's free {H}opf algebra construction revisited}, J.
  Pure Appl. Algebra \textbf{216} (2012), no.~8-9, 1768--1774. \MR{2925870}

\bibitem{vdb2017}
T.~Raedschelders and M.~Van~den Bergh, \emph{The {M}anin {H}opf algebra of a
  {K}oszul {A}rtin-{S}chelter regular algebra is quasi-hereditary}, Adv. Math.
  \textbf{{\bf 305}} (2017), 601--660. \MR{3570144}

\bibitem{Sierra2011}
S.~J. Sierra, \emph{{$G$}-algebras, twistings, and equivalences of graded
  categories}, Algebr. Represent. Theory \textbf{{\bf 14}} (2011), no.~2,
  377--390. \MR{2776790}

\bibitem{Takeuchi1971}
M.~Takeuchi, \emph{Free {H}opf algebras generated by coalgebras}, J. Math. Soc.
  Japan \textbf{{\bf 23}} (1971), 561--582. \MR{292876}

\bibitem{TVan}
H.~V. Tran and M.~Vancliff, \emph{Twisting systems and some quantum $p^3$s with
  point scheme a rank-2 quadric}, `Recent Advances in Noncommutative Algebra
  and Geometry,'' Eds. K. A. Brown et al, Contemporary Math., Amer. Math. Soc.,
  to appear.

\bibitem{WaltonWang2016}
C.~Walton and X.~Wang, \emph{On quantum groups associated to non-{N}oetherian
  regular algebras of dimension 2}, Math. Z. \textbf{{\bf 284}} (2016),
  no.~1-2, 543--574. \MR{3545505}

\bibitem{Zhang1996}
J.~J. Zhang, \emph{Twisted graded algebras and equivalences of graded
  categories}, Proc. London Math. Soc. (3) \textbf{{\bf 72}} (1996), no.~2,
  281--311. \MR{1367080}

\end{thebibliography}
\bibliographystyle{amsplain}
    
\end{document}